\documentclass[reqno,12pt]{amsart}
\usepackage[foot]{amsaddr}
\usepackage{macros}

\bibliography{references}

\title{A Completeness Theorem for Topological Doctrines}
\author{Silvio~Ghilardi and Jérémie~Marquès}
\address{Department of Mathematics, Universit\`a degli Studi di Milano, Italy}
\date{\today}

\begin{document}

\begin{abstract}
    We extend logical categories with fiberwise interior and closure operators so as to obtain an embedding theorem into powers of the category of topological spaces. The required axioms, besides the Kuratowski closure axioms, are  a ``product independence'' and a ``loop contraction'' principle.
\end{abstract}

\maketitle

\section{Introduction}

In this paper, we apply the methods of categorical logic to design a classical first-order logic with additional connectives for the interior and closure operators, which is sound and complete with respect to interpretations in topological spaces.
Whereas the propositional fragment of such a logic has been known for a long time~\cite{MckinseyTarski}, the extension to predicate logic has to face non trivial problems.
The most notable problem is that interior and closure do not commute with inverse image. Instead, they only \emph{semi}-commute: if $f : X\to Y$ is a continuous map between topological spaces, and if $\Diamond A$ represents the closure of $A$, then the inclusion $\Diamond f^{-1}(A) \subseteq f^{-1}(\Diamond A)$ holds, but equality fails in general, already in the case where $f$ is a diagonal function $X\longrightarrow X\times X$. This is especially problematic because inverse images represent substitutions from a logical point of view (in particular, inverse images along diagonals represent variable identifications in a formula).

In order to appropriately handle the above difficulties in designing a symbolic calculus, we switched to an approach based on categorical logic; in fact, categorical logic offers conceptual tools that are particularly illuminating for the context we are considering. At first glance, however, a  categorical approach seems to be problematic too, because  the category $\cTop$ of topological spaces fails to have the standard structure required in order to interpret first order logic: it is not a regular category. This difficulty is handled by replacing the regular-epi/mono factorization system of regular categories by an arbitrary stable factorization system: in  $\cTop$ such a system is given by the surjection/subspace factorizations (which are nothing but epi/regular-mono factorizations). The possibility of adopting stable factorization systems different from the standard regular-epi/mono factorization system traditionally used in categorical logic already appeared in~\cite{GhilardiZawadowski2011} and is fully exploited in
\cite{GM2025}, where the notions of a \emph{modal category} is introduced.

Modal categories (the main framework of this paper) are based on
\emph{f-Boolean} categories: these are ``Boolean categories relative to a stable factorization system,'' namely lex categories endowed with a proper and stable orthogonal factorization system $(\Ec,\Mc)$, whose lattices of $\Mc$-subobjects are Boolean algebras. They can be alternatively introduced as extensional Lawvere doctrines with full comprehension~\cite{MaiRosElementaryQuotientCompletion2013,Pasquali} and they can be turned into Boolean categories just by adding some isomorphisms (in the same way as $\cSet$ can be obtained, up to equivalence, from $\cTop$ by turning bijective continuous functions into iso's, see Section~\ref{sec:syntax-doctrines} below for more details).

A modal category $\cE$ is now simply defined as an f-Boolean category whose $\Mc$-subobject lattices are modal algebras, namely Boolean algebras endowed with a further operator $\Diamond$ commuting with finite joins. Since the $\Diamond$ operator is meant to represent the closure operator and the arrows of 
$\cE$ are meant to represent continuous functions, taking inverse images (i.e., pullbacks) along them only \emph{semi}-preserves $\Diamond$ in a modal category. There are many examples of modal categories beside $\cTop$: for instance, graphs and posets are modal categories. Graphs can be generalized to a simple notion of ``counterpart structure'' giving a complete semantics for modal categories~\cite{GM2025}.
In the current paper, we will see which additional axioms should be added to the notion of a modal category to get completeness relatively to the topological semantics.

In technical terms, we provide a \emph{complete axiomatization} of the modal categories that embed conservatively in a power of the category of topological spaces.
Our axioms include the S4 axiom $S \leq \Diamond S = \Diamond\Diamond S$, the \emph{product independence axiom} \eqref{ax:PI} as well as a \emph{loop contraction axiom} \eqref{ax:LC}.
The product independence axioms
says (in its logical formulation) that $\Diamond$ distributes over conjunctions of formulas not sharing common variables
\[ (\Diamond\varphi)[x] \land (\Diamond\psi)[y] 
\leftrightarrow
\Diamond(\varphi[x] \land \psi[y]) ~~.\]
The loop contraction axiom asserts a continuity condition for certain composable loops of partial maps. This last axiom is only necessary when certain function symbols are to be interpreted as subspace embeddings, because it is otherwise derivable from the rest of the axiomatic basis, see Remark~\ref{rmq:loop-contraction-trivializes} below (but notice that the formula \eqref{ax:LC} plays nevertheless an important role in the completeness proof even when the language is restricted so as  it becomes derivable). 

\paragraph{Related work} Several authors in categorical logic have considered first-order S4 modal logic, for instance \cite{ReyToposTheoreticApproachReference1991,MakReyCompletenessResultsIntuitionistic1995} or \cite{AwoKisTopologyModalityTopological2008}. Modalities not commuting with substitutions were introduced in~\cite{GhiMelModalTensePredicate1988}, where  completeness is proved with respect to presheaves and relational presheaves (the completeness  results in~\cite{GhiMelModalTensePredicate1988} however do not apply to languages/doctrines with sorts representing subspaces). 
In \cite[Cap.~II, \S~7]{GhiModalitaCategorie1990}  interpretations in topological spaces  were considered too: the product independence axiom was introduced and completeness for purely relational languages was proved. These partial results remained unpublished (they were just announced in~\cite{viareggio}); in the present paper, the topological construction from \cite[Cap.~II, \S~7]{GhiModalitaCategorie1990} is improved (in particular by introducing `lax' features in it), so as it becomes powerful enough to prove the completeness result for full languages with the help of the completely novel loop contraction formulas.

The only attempt to axiomatize  first-order S4 modal logic in topological spaces we are aware of is in \cite{SgrInteriorOperatorLogic1980}; actually, \cite{SgrInteriorOperatorLogic1980} allows the interior and closure operators to be applied with respect to a subset of the free variables of a formula, unlike  in the present paper. The axiomatization obtained in \cite{SgrInteriorOperatorLogic1980} 
looks very involved when compared with the neat axiomatization of the present paper  (the relationship and the connections of the two axiomatizations are obscure and a priori not obvious).

\paragraph{Overview of the paper} In \S~\ref{sec:prelim}, we recall the notion of modal category from \cite{GM2025} and some associated terminology. In \S~\ref{sec:top-ax}, we define three axiom schemas and show that they are satisfied in the modal category $\cTop$ of topological spaces. Based on these axioms, we define topological modal categories in \S~\ref{sec:completeness} and we show a completeness result relatively to $\cTop$. In \S~\ref{sec:syntax-doctrines}, we give a more concrete meaning to this result by building the classifying topological modal category of a first-order modal theory. This construction can be done either directly or in two steps: from a theory to a suitable Lawvere doctrine and from a Lawvere doctrine to a suitable logical category. Since the first step is rather standard, we detail here only the second step
(in~\cite[\S 7]{GM2025} we presented a direct construction, but it should be noted that it is \emph{not} the syntactic category of a modal theory for the reasons explained in~\cite[\S 8.1]{GM2025}).

\paragraph{Notations} In a topological space $X$, we denote by $\ovl{A}$ the closure of a subset $A\subseteq X$. Given two objects $X,Y$ in a category with products, we denote by $\pi_X : X \times Y \to X$ the canonical projection. The composite of $f : X \to Y$ and $g : Y \to Z$ in a category will be written $fg$ or $g \circ f$ (often, $f$ and $g$ will be relations or partial maps). The adjective ``lex'' refers to the existence or preservation of finite limits. The adjective ``cartesian'' refers to the existence or preservation of finite products. Given an element $x$ of a poset, the notation ${\downarrow}x$ denotes the set of elements that are below $x$.

\section{Preliminaries}
\label{sec:prelim}

\subsection{Modal categories}\label{subsec:modalcategopries}

We recall in this section some material from \cite{GM2025}. Beware that while the setting of \cite{GM2025} 
does not require the presence of classical negation, we will use only \emph{Boolean} modal logic in the current paper.

We first summarize the motivation for the notions to be recalled. In many of the usual structures of categorical logic, such as Boolean categories \cite[A1.4, p.38]{JohSketchesElephantTopos2002}, the whole logical structure is determined by the bare category without any extra data. For instance, the subobjects coincide with the monomorphisms and conjunctions are realized as pullbacks. However, we should not hope the same thing to be possible for topological logics and more generally for
modal logics. For instance, in many categories with a ``modal'' structure such as graphs or topological spaces, not every monomorphism is an embedding (i.e., a subgraph or a subspace). In general, we must specify \emph{which} monomorphisms are ``embeddings'' and how the modalities act on these subobjects. Since we want to include existential quantification in the internal language, every map must also have an ``image.'' As a consequence, the class of embeddings should be the right class of an orthogonal factorization system. We thus start by relaxing the notion of Boolean category, with the help of a factorization system. More details can be found in \cite{GM2025}.

\begin{definition}
	An \emph{f-Boolean category} is a lex category $\cE$ equipped with an orthogonal factorization system $(\Ec,\Mc)$ such that:
	\begin{itemize}
		\item $\Mc$ is included in the class of monomorphisms and contains every regular monomorphism. Such factorization systems are called \emph{proper}.
		\item $\Ec$ and $\Mc$ are stable under pullbacks (this needs not to be required  for $\Mc$ because it  is automatic). Such factorization systems are called \emph{stable}.
		\item For each $X \in \cE$, the poset of $\Mc$-subobjects is a (small) Boolean algebra. This modal algebra is denoted by $\Sub_\cE(X)$ or simply $\Sub(X)$.
		\item For each arrow $f : X\to Y$ in $\Ec$, the pullback map $f^* : \Sub_\cE(Y) \to \Sub_\cE(X)$ is a morphism of Boolean algebras.
	\end{itemize}
	A morphism of f-Boolean categories is a lex functor which preserves the factorization system and the join of $\Mc$-subobjects.
\end{definition}

\paragraph{Notations} If $\cE$ is an f-Boolean category, the arrows in $\Mc$ are called \emph{embeddings} and the arrows in $\Ec$ are called \emph{surjections}. The isomorphism classes of embeddings into $X$ are called the \emph{subobjects} of $X$. Since this convention is not standard in category theory where a ``subobject'' is understood as a monomorphism, we justify our choice by remarking that in general, the notion of subobject depends on the situation. For instance, a subspace of a topological space $X$ is not a monomorphism into $X$. We will also write $A \subseteq X$ instead of $A \in \Sub(X)$. If $f : X\to Y$ is a morphism in $\cE$, its \emph{image} is the subobject of $Y$ obtained by factorizing $f$ according to the factorization system. If $A \subseteq X$, we denote by $\exists_f(A) \subseteq Y$ the image of the composite $A \hookrightarrow X \to Y$. The map $\exists_f : \Sub(X) \to \Sub(Y)$ is left adjoint to $f^* : \Sub(Y) \to \Sub(X)$. The pullback of a subobject $A$ by some morphism $f$ will be written $f^* A$ or $f^{-1} A$. The direct image will be written $\exists_f A$ or $f A$. For more details, we refer to \cite[Sec.~2]{GM2025}. These conventions allow us to leave $\Ec$ and $\Mc$ implicit, lightening the notation.

The next step is to add modalities. A \emph{modal algebra} is a Boolean algebra $M$ equipped with an operator $\Diamond : M\to M$ preserving finite joins, or equivalently an operator $\Box : M\to M$ preserving finite meets with $\Box = \lnot\Diamond\lnot$. A \emph{lax morphism} between two modal algebras $M$ and $N$ is a morphism $f : M\to N$ of Boolean algebras such that $\Diamond f(x) \leq f(\Diamond x)$.

\begin{example}
	If $X$ is a topological space, then $\Ps(X)$ is a modal algebra with $\Diamond A = \ovl{A}$ the closure of $A$. If $f : X\to Y$ is a continuous map, then $f^{-1} : \Ps(Y) \to \Ps(X)$ is a lax morphism.
\end{example}

\begin{definition}\label{dfn:mod-cat}
	A \emph{Boolean modal category} is an f-Boolean category $\cE$ such that:
	\begin{itemize}
		\item $\Sub(X)$ has a structure of modal algebra for each $X \in \cE$.
		\item For each morphism $f : X\to Y$, the map $f^* : \Sub(Y)\to\Sub(X)$ is a lax morphism:
		\[ \Diamond f^* A \leq f^* \Diamond A \text{.} \]
		\item For each embedding $m : X\hookrightarrow Y$ and for each $A \subseteq X$:
		\[ \Diamond A = m^* \Diamond \exists_m A \text{.} \]
	\end{itemize}
	A morphism of Boolean modal categories is a morphism of the underlying f-Boolean categories which commutes with the modalities.
\end{definition}

Since we will only consider Boolean modal categories in this paper, we will simply write ``modal category'' instead of ``Boolean modal category.''

\begin{example}
	The category $\cTop$ of topological spaces is a modal category, with the factorization system of surjections and embeddings 
    (which are just epis and regular monos).
    The modality $\Diamond$ is the closure operator and $\Box$ is the interior operator.
\end{example}

\paragraph{Partial maps} In a modal category, a \emph{partial map} $f : X\relto Y$ is given by a subobject $d_f : \dom(f) \hookrightarrow X$, and a morphism $t_f : \dom(f)\to Y$. Partial maps can be composed by using that subobjects are stable under pullbacks. The direct and reciprocal images of $A$ under $f$ are given by the formulas $f^{-1} A = d_f t_f^{-1} A$ and $f A = t_f d_f^{-1} A$. The \emph{graph} of $f$ is $(d_f,t_f) : \dom(f) \to X\times Y$.

\begin{lemma}
	The graph of a partial map $f$ is an embedding.
\end{lemma}

\begin{proof}
	The map $(\id,t_f) : \dom(f) \to \dom(f)\times Y$ is an embedding since it is the equalizer of $f\circ\pi_{\dom(f)}$ and $\pi_Y$. The map $d_f\times Y : \dom(f)\times Y \to X\times Y$ is also an embedding since it is the pullback of $d_f$ along $\pi_X : X\times Y\to X$. The composite of $(\id,t_f)$ and $d_f\times Y$ is $(d_t,t_f)$, so it is an embedding.
\end{proof}

By virtue of the previous lemma, a partial map can also be thought of as a special relation,
i.e. a subobject of a cartesian product 
$R \subseteq X\times Y$. The composition of partial maps coincides with the composition of relations. For this reason we will not distinguish a partial map and its graph. In general, when $R \subseteq X\times Y$ is a relation, we denote by $R A = \pi_Y[R \land \pi_X^{-1} A]$ the direct image of $A \subseteq X$ under $R$ and we denote by $R^{-1} \subseteq Y\times X$ the reciprocal relation. This is consistent with the notations $f A$ and $f^{-1} A$ for direct and reciprocal image under a partial map $f$.

When $f$ is a total map, it satisfies a continuity axiom that can be expressed in many ways such as $f \Diamond A \leq \Diamond f A$ or $f \Diamond f^{-1} A \leq \Diamond A$, thanks to the adjunction relation $f[{-}] \dashv f^{-1}({-})$. When $f$ is only a partial map, this adjunction relation does not hold and the correct formulation of continuity becomes
\begin{equation*}
	f \Diamond f^{-1} A \leq \Diamond A \text{.}
\end{equation*}
Indeed, this equation can be obtained by combining the continuity axiom for $t_f$ (i.e., $t_f \Diamond t_f^{-1} B \leq \Diamond B$) and the embedding axiom for $d_f$ (i.e., $d_f^{-1} \Diamond d_f B \leq \Diamond B$).

\begin{remark}
	More generally, a relation $R \subseteq X\times Y$ is continuous when it satisfies $R \Diamond R^{-1} A \leq \Diamond A$. The completeness result presented here seems to generalize straightforwardly when partial maps are replaced by continuous relations, but we leave that aside in this paper.
\end{remark}

\subsection{Topological semantics}

In the same way that a small Boolean category encodes a Boolean first-order theory, a small modal category encodes a kind of modal first-order theory
(see \cite{GM2025} for details). If $\cC$ is a small Boolean category, the models of the associated theory are the coherent functors $\cC \to \cSet$. As already mentioned, several semantic categories can play the role of $\cSet$ for modal logic: graphs, topological spaces, posets... For instance:
\begin{definition}\label{dfn:top-mod}
	A \emph{topological model} of a modal category $\cC$ is a modal functor $\cC \to \cTop$.
\end{definition}

However, none of the aforementioned semantic categories provide a complete semantics for small modal categories. Indeed, they satisfy additional axioms not required by Definition~\ref{dfn:mod-cat}. In topological spaces and posets, the S4 axiom $A \leq \Diamond A = \Diamond\Diamond A$ holds. In graphs, topological spaces and posets, the product independence axiom \eqref{ax:PI}, to be introduce later, holds. Nonetheless, the interested reader can find a complete semantics for small modal categories in \cite{GM2025}.

In this paper, we focus on the \emph{topological} semantics of Definition~\ref{dfn:top-mod}. We will see which additional axioms should be added to modal categories so that this semantics becomes complete (and stays sound). To formulate that more precisely, we introduce some terminology.

A modal functor $F : \cE \to \cC$ is called \emph{conservative} if $\Sub_\cE(X) \to \Sub_\cC(F(X))$ is an embedding for each $X \in \cE$. In logical terms, the idea is that the theory of $\cC$ is a conservative extension of the theory of $\cE$. Conservative modal functors are always faithful. More generally, a potentially large family $(F_i : \cE \to \cC_i)_i$ of modal functors is \emph{jointly conservative} whenever $\Sub_\cE(X) \to \prod_i \Sub_{\cC_i}(F_i(X))$ is an embedding for each $X \in \cE$ (the lattice $\prod_i \Sub_{\cC_i}(F_i(X))$ is large if the family is large but this is not a problem).

\paragraph{Main result} The main result of this paper (Theorem~\ref{thm:completeness-top}) is a characterization of the small modal categories $\cE$ such that the family of modal functors $\cE \to \cTop$ is jointly conservative. Equivalently, since $\cE$ contains only a set of embeddings, this means that there is a conservative functor to a power of $\cTop$ indexed by a set.

Just like in \cite{GM2025}, the completeness result of this paper is shown using the models of the underlying f-Boolean category:

\begin{definition}
	A \emph{Boolean model} of a modal category $\cE$ is an f-Boolean functor $\cE\to\cSet$.
\end{definition}

\paragraph{Notations} Given a Boolean model $M : \cE \to \cSet$, we will often write $\llbracket\varphi\rrbracket_M$ instead of $M(\varphi)$ when $\varphi \subseteq X$ is a subobject in $\cE$. When $x \in M(X)$, we will write $x \vDash \varphi$ to mean that 
$x \in \llbracket\varphi\rrbracket_M=M(\varphi) \subseteq M(X)$.

\section{The topological axioms}
\label{sec:top-ax}

In this section, we introduce three axiom schemas and show that they are satisfied in $\cTop$: the closure axiom \eqref{ax:S4}, the product independence axiom \eqref{ax:PI} in \S~\ref{subsec:product-indep}, and the loop contraction axiom \eqref{ax:LC} in \S~\ref{subsec:loop-contract}. The modal categories satisfying these axioms will be called \emph{topological} and shown to embed conservatively in a power of $\cTop$.

The product independence and the loop contraction axioms are easier to understand when all the spaces involved are Alexandroff. We encourage the reader to consider this special case, noting that the full modal subcategory of $\cTop$ spanned by the Alexandroff spaces is equivalent to the modal category of posets, where $\Diamond$ is the downward closure operator.

First of all, the S4 axiom
\begin{equation}\label{ax:S4}
	S \leq \Diamond S = \Diamond\Diamond S \tag{S4}
\end{equation}
simply expresses that $\Diamond$ is a closure operator, as is indeed the case in $\cTop$.

\subsection{Product independence}
\label{subsec:product-indep}

Let $\cE$ be a modal category and let $X, Y$ be two objects of $\cE$. The \emph{product independence axiom} states that for all $A \subseteq X$ and $B \subseteq Y$,
\begin{equation}\label{ax:PI}
	(\pi_X^{-1} \Diamond A) \land (\pi_Y^{-1} \Diamond B) \leq \Diamond (\pi_X^{-1} A \land \pi_Y^{-1} B) \text{.} \tag{PI}
\end{equation}

Note that the reverse inequality holds in any modal category.

\begin{lemma}\label{lem:PI-valid}
	The axiom \eqref{ax:PI} holds in $\cTop$.
\end{lemma}

\begin{proof}
	Let $A$ and $B$ be two topological spaces. Let $\varphi \subseteq A$ and let $\psi \subseteq B$. We need to show that $\ovl{\varphi\times\psi} = \ovl{\varphi}\times\ovl{\psi}$, where $\ovl{\theta}$ is the closure of $\theta$. Let $(a,b) \in \ovl{\varphi\times\psi}$. It means that for all open $U \subseteq A$ and all open $V \subseteq B$, we have $(U\times V) \cap (\varphi\times\psi) \neq \emptyset$, or in other words that $U \cap \varphi \neq \emptyset$ and that $V \cap \psi \neq \emptyset$. So it is equivalent to $a \in \ovl{\varphi}$ and $b \in \ovl{\psi}$.
\end{proof}

\begin{remark}\label{rmq:proj-open}
	The axiom \eqref{ax:PI}, in conjunction with the axiom \eqref{ax:S4}, implies in particular that projections are ``open,'' meaning that reindexing along them commutes with the modalities: $\pi_X^{-1} \Diamond A = \Diamond \pi_X^{-1} A$. To see that, take $B=\top$ and use that $\top \leq \Diamond\top$. This property is useful to translate categorical notation to logic notation, where reindexing along projections is omitted.
    (But beware that reindexation along diagonals, which corresponds to duplicating a variable in the syntax of first-order logic, still does not commute with the modalities.)
\end{remark}

\subsection{Loop contraction}
\label{subsec:loop-contract}

As we have seen, the correct continuity axiom for a \emph{partial} map $f : X \relto Y$ is
\begin{equation*}
	f \Diamond f^{-1} A \leq \Diamond A \text{.}
\end{equation*}
We will need a more general axiom schema involving several partial maps. When only total maps are involved, we will see in Remark~\ref{rmq:loop-contraction-trivializes} that these axioms trivialize thanks to the equivalence of $f \Diamond f^{-1} A \leq \Diamond A$ and $f \Diamond A \leq \Diamond f A$.

In what follows, a \emph{loop} in a modal category is a sequence of binary relations $R_1,{\ldots},R_n$ such that the composite $R_1 \circ \cdots \circ R_n$ makes sense 
(i.e. is type-matching)
and such that the codomain of $R_1$ coincides with the domain of $R_n$. The object which is both the codomain of $R_1$ and the domain of $R_n$ is called the \emph{anchor} of the loop. To each loop $R_1,{\ldots},R_n$, we associate the following ``loop contraction'' condition, where $A$ ranges over the subobjects of the anchor:
\begin{equation}\label{ax:LC}
    \Diamond R_1 \Diamond R_2 \Diamond \cdots \Diamond R_n A \leq \Diamond A \text{.} \tag{LC}
\end{equation}
Note that the first $\Diamond$ can be removed in \eqref{ax:LC} without affecting the statement, thanks to axiom \eqref{ax:S4}. As a consequence, $f \Diamond f^{-1} A \leq \Diamond A$ is indeed a special case of \eqref{ax:LC}.

We define by induction the \emph{acceptable} loops:
\begin{enumerate}[label=(A\arabic*)]
	\item the empty loop is acceptable (on any anchor), \label{acc:1}
	\item identities can be added anywhere in acceptable loops to produce new acceptable loops,\label{acc:2}
	\item if $w$ and $w'$ are acceptable, then $w,w'$ is acceptable,\label{acc:3}
	\item if $w$ is acceptable and if $f$ is a partial map, then $f,w,f^{-1}$ is acceptable,\label{acc:4}
	\item if $R_1,{\ldots},R_n$ and $S_1,{\ldots},S_n$ are acceptable, then $R_1\times S_1,{\ldots},R_n\times S_n$ is acceptable. \label{acc:5}
\end{enumerate}

\begin{example}
	If $f_1 : X_1\relto Y_1$ and $f_2 : X_2\relto Y_2$ are partial morphisms, then
	\[ f_1\times Y_2, X_1\times f_2, f_1^{-1}\times X_2, Y_1\times f_2^{-1} \]
	is an acceptable loop whose anchor is $Y_1\times Y_2$.
\end{example}

\begin{proposition}\label{prop:PE-valid}
	The axiom \eqref{ax:LC} holds for all the acceptable loops in $\cTop$.
\end{proposition}

\begin{proof}
	We will show that the loops $R_1,{\ldots},R_n$ in $\cTop$ such that \eqref{ax:LC} holds are closed under \ref{acc:1}--\ref{acc:5}. It is quite direct for \ref{acc:1}--\ref{acc:4}, as it is a general fact of modal categories. In the rest of the proof, we check the stability under \ref{acc:5}. Let $R_1,{\ldots},R_n$ and $S_1,{\ldots},S_n$ be two loops in $\cTop$ with respective anchors $X$ and $Y$. Suppose that \eqref{ax:LC} holds for both of these loops. We need three observations:
	\begin{itemize}
		\item If $\pi_V : V\times W\to V$ is a product projection, then $\pi_V^{-1} \Diamond I = \Diamond \pi_V^{-1} I$ for any $I \subseteq V$.
		\item If $R \subseteq V\times V'$ and $S \subseteq W\times W'$, then $(R\times S) \pi_V^{-1} I \leq \pi_{V'}^{-1} R I$ for any $I \subseteq V$.
		\item The map $A \subseteq X\times Y \mapsto \Diamond(R_1\times S_1)\cdots\Diamond(R_n\times S_n) A$ preserves finite unions.
	\end{itemize}
	Given $I \subseteq X$ and $J \subseteq Y$, we write $I \oplus J$ instead of $\pi_X^{-1}(I) \lor \pi_Y^{-1}(J)$. Using the three facts cited above, we obtain the following formula:
	\[ \Diamond (R_1\times S_1) \cdots \Diamond (R_n\times S_n) (I\oplus J) \leq (\Diamond R_1 \cdots \Diamond R_n I) \oplus (\Diamond S_1 \cdots \Diamond S_n J) \text{.} \]
	By definition of the product topology on $X\times Y$, we have
	\[ \Diamond A = \bigcap_{A\subseteq I\oplus J} \Diamond I \oplus \Diamond J \]
	for all $A \subseteq X\times Y$. We obtain for all $A \subseteq X\times Y$ the following chain of inclusions:
	\begin{align*}
		\Diamond A &= \bigcap_{A\subseteq I\oplus J} \Diamond I \oplus \Diamond J\\
		&\geq \bigcap_{A\subseteq I\oplus J} (\Diamond R_1 \cdots \Diamond R_n I) \oplus (\Diamond S_1 \cdots \Diamond S_n J) \\
		&\geq \bigcap_{A\subseteq I\oplus J} \Diamond (R_1\times S_1) \cdots \Diamond (R_n\times S_n) (I\oplus J) \\
		&\geq \Diamond (R_1\times S_1) \cdots \Diamond (R_n\times S_n) A\text{.}
	\end{align*}
	This concludes the proof that \eqref{ax:LC} holds for $R_1\times S_1,{\ldots},R_n\times S_n$.
\end{proof}

\begin{remark}\label{rmq:loop-contraction-trivializes}
	When all the maps involved in an acceptable loop are total, the associated loop contraction axiom \eqref{ax:LC} is in fact deducible from the simple continuity axiom $f\Diamond A \leq \Diamond fA$ where $f$ is a \emph{total} maps. We will write the proof using a trick to facilitate its formal description, but the underlying idea is that of rewriting systems. An \emph{auxiliary sequence} of a loop $R_1,{\ldots},R_n$ is a sequence $G_0,G_1,{\ldots},G_n$ whose terms are graphs of \emph{total} maps such that $G_k \circ R_{k+1}\leq G_{k+1}$ for all $0\leq k\leq n$, and $G_0=G_n=\Id$, where $\Id$ is the graph of the identity of the anchor. We show that every acceptable loop in which only total maps appear has an auxiliary sequence by induction on the construction rules \ref{acc:1}--\ref{acc:5}. For \ref{acc:1}, this is simply the one-element sequence $\Id$. For \ref{acc:2}, we duplicate one of the $G_k$ depending on where an identity has been added. For \ref{acc:3}, we concatenate the two auxiliary sequences, merging the final element of the first with the initial element of the second (they are both the graph of the identity). For \ref{acc:4}, if $G_0,G_1,{\ldots},G_n$ is an auxiliary sequence of $R_1,{\ldots},R_n$ and if $f$ is a total map, then $\Id,f\circ G_0,f\circ G_1,{\ldots},f\circ G_n,\Id$ is an auxiliary sequence of $f,R_1,{\ldots},R_n,f^{-1}$, using that $f\circ f^{-1} \leq \Id$. For \ref{acc:5}, we take the pointwise product of the auxiliary sequences. To conclude, we show that if $R_1,{\ldots},R_n$ has an auxiliary sequence $G_1,{\ldots},G_n$, then the formula
	\[ \Diamond R_1 \cdots \Diamond R_n A \leq \Diamond A \]
	can be derived from the continuity axiom $f \Diamond A \leq \Diamond f A$ where $f$ is a \emph{total} map:
	\begin{align*}
		\Diamond R_1 \cdots \Diamond R_n A &= \Diamond G_0 R_1 \Diamond R_2 \cdots \Diamond R_n A \\
		&\leq \Diamond G_1 \Diamond R_2 \cdots \Diamond R_n A \\
		&\leq \Diamond \Diamond G_1 R_2 \cdots \Diamond R_n A \tag{by the continuity axiom applied to $G_1$} \\
		&\leq \Diamond \Diamond G_2 \cdots \Diamond R_n A \\
		&\leq \cdots \leq \Diamond^n G_n A \leq \Diamond A\text{.}
	\end{align*}
\end{remark}

\section{Topological modal categories}
\label{sec:completeness}

\begin{definition}
	A modal category $\cE$ is called \emph{topological} when it satisfies the axiom \eqref{ax:S4}, the product independence axiom \eqref{ax:PI}, and the loop contraction axiom \eqref{ax:LC} for all the acceptable loops in $\cE$.
\end{definition}

We can now state the main result of this paper:

\begin{theorem}\label{thm:completeness-top}
	A small modal category is topological if and only if it admits a conservative modal functor to a power of $\cTop$.
\end{theorem}

We have seen in \S~\ref{sec:top-ax} that $\cTop$ is a topological modal category. This in fact shows half of Theorem~\ref{thm:completeness-top} (the semantics is sound):

\begin{proposition}\label{prop:top-soundness}
	If a modal category admits a conservative modal functor to a power of $\cTop$, then it is topological.
\end{proposition}

\begin{proof}
	Suppose that $\cE$ is a modal category and let $(F_i : \cE \to \cTop)_{i\in I}$ be a conservative family of modal functors (or equivalently a single conservative modal functor $\cE \to \cTop^I$). We only show that $\cE$ satisfies the loop contraction axiom, the other axioms being similar. Each acceptable loop in $\cE$ is sent to an acceptable loop by each $F_i$. Hence the image by each $F_i$ of the inequality $\eqref{ax:LC}$ holds for each acceptable loop, and since the family $(F_i)_i$ is jointly conservative, $\eqref{ax:LC}$ holds in $\cE$.
\end{proof}

The rest of this section is devoted to proving the other direction of Theorem~\ref{thm:completeness-top}.

\subsection{Constructing topological models}
\label{subsec:top-mod}

We present a slightly more general version of the construction of \cite[Cap.~II, \S~7]{GhiModalitaCategorie1990} to build topological models. The same thing would work for filtered colimits, but we only need it for $\omega$-indexed colimits.

A \emph{relational sequence} $(X,R)$ is a sequence  of sets and functions
$X_1 \to X_2 \to X_3 \to \cdots$
equipped with relations $R_{ij} \subseteq X_i \times X_j$ for each $i \leq j$ such that:
\begin{itemize}
	\item $R_{ij} R_{jk} \subseteq R_{ik}$ for all $i \leq j \leq k$,
	\item $R_{ij}$ contains the graph of the map $X_i \to X_j$ for all $i\leq j$.
\end{itemize}
An \emph{element} of a relational sequence $(X,R)$ is an element  $x$
of some of the $X_k$, and we simply write $x \in X$
to express this.
We also simply write $x R y$ instead of $x R_{ij} y$ when this creates no confusion. A morphism between two relational sequences $(X,R)$ and $(Y,S)$ is a natural transformation $f : X\to Y$ between the underlying sequences of sets and functions 
which preserves the relations: $x R y \implies f(x) S f(y)$. We denote by $\cSetRel^\omega$ the category of relational sequences. This notation is justified by the fact that its objects are the lax natural transformations from $\omega$ to a certain category $\cSetRel$.

Given a relational sequence $(X,R)$, we build a topological space $\colim(X,R)$ as follows. The underlying set is the colimit $c(X) = \colim_i X_i$. For all $p \in X$, we denote by $[p]$ its image in $c(X)$. For all $p \in X$, we define
\[ I_p \coloneqq \setst{[q]}{p R q} \subseteq c(X) \]
so that $[r] \in I_p \iff \exists q \in X : [r] = [q] \land p R q$.
We equip $c(X)$ with the topology whose open subsets are generated by the $I_p$.

\begin{lemma}
	For each $x \in c(X)$, $\setst{I_p}{x=[p]}$ is a basis of open neighborhoods of $x$.
\end{lemma}

\begin{proof}
	We first verify that $I_p$ contains $x$ if $x = [p]$. This is because 
    $R$ contains the graph of  $X_i\longrightarrow X_{i+1}$, 
    so that $x = [p] \in I_p$. We also see that if $p R q$, then $I_q \subseteq I_p$. Indeed, if $q R r$ then $p R q R r$, so that $p R r$.
	
	Because of general topological facts, we need to show two things to conclude:
	\begin{itemize}
		\item the set of designated neighborhoods of $x$ is filtered;
		\item if $[p]=x$ and $y \in I_p$, then there is $q$ with $[q]=y$ and $I_q \subseteq I_p$.
	\end{itemize}
	
	For the first statement, let $p \in X$ and $q \in X$ such that $[p] = [q] = x$. This means that there is some index $k$ big enough such that the images of $p$ and $q$ in $X_k$ are equal to the same point $r$. As a consequence, $p R r$ and $q R r$, which shows that $I_r \subseteq I_p \cap I_q$.
	
	For the second statement, let $p \in X$ and suppose that $y \in I_p$. By definition of $I_p$, this means that there is $q \in X$ such that $p R q$ and $[q] = y$. Then $I_q \subseteq I_p$.
\end{proof}

\begin{example}
	Both posets and metric spaces occur as special cases of this construction. If $X$ is a poset, we put $X_i = X$ for all $i$ and $x R_{ij} y$ iff $x \leq y$ 
    (the maps $X=X_i\longrightarrow X_{i+1}=X$ are the identity). 
    If $X$ is a metric space, we also put $X_i = X$ for all $i$ 
    (the transition maps are again identities) 
    and $x R_{ij} y$ iff $d(x,y) \leq 1/i - 1/j$.
\end{example}

\begin{lemma}
	The construction above defines a functor $c : \cSetRel^\omega \to \cTop$.
\end{lemma}

\begin{proof}
	Let $f : (X,R)\to(Y,S)$ be a morphism in $\cSetRel^\omega$. We must show that the induced map $c(f) : c(X) \to c(Y)$ is continuous. Indeed, for any $p \in X$, we have $c(f)(I_p) \subseteq I_{f(p)}$ since $pRq \implies f(p)Rf(q)$.
\end{proof}

An \emph{embedding} in $\cSetRel^\omega$ is a morphism $(X,R) \to (Y,S)$ which is pointwise injective and such that $R$ is the restriction of $S$ to $X \subseteq Y$.

\begin{lemma}\label{lem:c-embed}
	The functor $c : \cSetRel^\omega \to \cTop$ sends embeddings to embeddings.
\end{lemma}

\begin{proof}
	Let $f : (X,R) \hookrightarrow (Y,S)$ be an embedding in $\cSetRel^\omega$. Since a filtered colimit of injections in $\cSet$ is an injection, $c(f)$ is an injection $c(X) \hookrightarrow c(Y)$. We show that for any $p \in X$, the open subset $I_p \subseteq c(X)$ is the restriction of $I_{f(p)} \subseteq c(Y)$. Indeed, for $r \in X$, we have
	\begin{align*}
		[r] \in I_p &\iff \exists q\in X : [r]=[q] \land pRq\\
            &\iff \exists q\in X : [f(r)]=[f(q)] \land f(p)Rf(q)\\
		&\iff \exists q\in Y : [f(r)]=[q] \land f(p)Sq\\
		&\iff [f(r)] \in I_{f(p)} \text{.}
	\end{align*}
	The middle implication is due to the fact that if $[q] = [f(r)]$, then the image of $q$ in $Y_k$ is in $f(X_k) \subseteq Y_k$ for $k$ big enough.
\end{proof}

\begin{lemma}\label{lem:c-is-lex}
	The functor $c : \cSetRel^\omega \to \cTop$ preserves finite limits.
\end{lemma}

\begin{proof}
	We start with finite products. Finite products in $\cSetRel^\omega$ are computed pointwise: the product of $(X,R)$ and $(Y,S)$ is the sequence $(X_i\times Y_i)_i$ with relations $(R_{ij}\times S_{ij})_{i\leq j}$. Note that $\colim_i X_i\times Y_i = (\colim_i X_i) \times (\colim_i Y_i)$, so that $c(X\times Y) \to c(X)\times c(Y)$ is a continuous bijection. We need to show that the inverse function is also continuous. For all $(p,q) \in X_i\times Y_i$, we have $I_{(p,q)} = \setst{[(p',q')]}{(p,q) \mathbin{R\times S} (p',q')} = \setst{([p'],[q'])}{pRp' \land qRq'} = I_p \times I_q$. As a consequence, the inverse map $c(X)\times c(Y)\to c(X\times Y)$ is continuous.
	
	Finally, we treat equalizers. An equalizer in $\cSetRel^\omega$ is an embedding and is thus sent to an embedding in $\cTop$ by Lemma~\ref{lem:c-embed}. Combining this with the fact that equalizers commute with filtered colimits in $\cSet$, we obtain that $c : \cSetRel^\omega \to \cTop$ commutes with equalizers.
\end{proof}

\subsection{Completeness of the topological semantics}

This section is dedicated to proving that the topological semantics is complete for topological modal categories, i.e., to complete the proof of our main Theorem~\ref{thm:completeness-top}. We fix a topological modal category $\cE$. Given two functors $M,N : \cE \to \cSet$, a \emph{relation} $R \subseteq M\times N$ is a family of subsets $R(X) \subseteq M(X)\times N(X)$. It is \emph{stable under products} if $(x,y) R (x',y')$ holds  whenever $x R y$ and $x' R y'$. It is \emph{stable under partial maps} if for all partial maps $f : X\relto Y$ in $\cE$ and all $(x,y) \in M(X)\times N(X)$ such that $x R y$, if both $x$ and $y$ are in the domain of $f$, then $f(x) R f(y)$.
Notice that stability under partial maps implies both stability under total maps (``if $f:X\longrightarrow Y$ is an arrow in $\cE$ and $x R y$ then $f(x) R f(y)$'') and  under embeddings (``if $f:X\hookrightarrow Y$ is an embedding in $\cE$ and $f(x) R f(y)$ then $x R y$'').

\subsubsection{The Extension Lemma}

The Lemma below will be used as a basic building block for our completeness theorem.

\begin{lemma}\label{lem:extend-model}
	For each Boolean model $M : \cE \to \cSet$, there is another Boolean model $N : \cE \to \cSet$, a natural transformation
    $M \to N$ and a relation $R \subseteq M\times N$ which is stable under products, contains the graph of $M \to N$ and such that $\llbracket\Diamond\varphi\rrbracket_M = R^{-1} \llbracket\varphi\rrbracket_N$ for each subobject $\varphi \subseteq X$ in $\cE$.
\end{lemma}

We supply two proofs of this Lemma. The first proof is mostly sketched,  uses a translation into logical languages and has the merit of focusing intuition to the main point of the construction. The second proof is purely categorical and makes uses of the representation of Robinson diagrams as filtered pseudo-colimits. A comparison of the two proofs (which are the same but differ in the involved technologies) might be instructive. 

\begin{proof}[First Proof of Lemma~\ref{lem:extend-model}]
	We can always associate with an f-Boolean category $\cE$ a first order classical theory $\cE_T$: the language of such a theory includes a function symbol for every arrow of $\cE$, a predicate symbol for every subobject  in $\cE$, etc.
	The axioms of the theory comprise all formulas which are `internally valid' in $\cE$ (when $\cE$ is a modal category, this method treats modalized formulas as atomic predicates). The statement of the Lemma asks for a suitable elementary extension of a model $M$ of $\cE_T$. As usual in model theory, elementary extensions of $M$ are found via models of the elementary diagram $\Delta(M)$ of $M$. In our case, we need to expand $\Delta(M)$ with the
	sets of sentences $K_1$ and $K_2$ introduced below.
	
	Let us call \emph{basic pairs} the pairs $(\ua,\varphi(\uy))$ given by a formula $\varphi$ with free variables among $\uy=y_1, \dots, y_n$ (of appropriate sorts) and a matching tuple of elements $\ua\in M$ such that $M\models (\Diamond \varphi)(\ua)$.
	For every  basic pair  $(\ua,\varphi(\uy))$,
	we introduce a tuple $\underline{\kappa}^{\ua}=\kappa^{\ua}_1, \dots, \kappa^{\ua}_n$ of \emph{fresh} constants (they depend on $\varphi$ but we omit $\varphi$ from the notation).
	Now
	the set $K_1$ contains the sentences of the kind $\varphi(\underline{\kappa}^{\ua})$, where $(\ua,\varphi(\uy))$ is a basic pair, whereas
	the set $K_2$ contains for every tuple of basic pairs $(\ua_1,\varphi_1(\uy_1)), \dots, (\ua_m,\varphi_m(\uy_m))$
	and for every tuple $\uc\in M$ the sentences of the kind $B(\uc,\underline{\kappa}^{\ua_1},\dots,\underline{\kappa}^{\ua_m})$
	such that $M\models (\Box B)(\uc,{\ua_1},\dots,{\ua_m})$.
	
	To prove the Lemma we first show that set $\Delta(M)\cup K_1\cup K_2$ is consistent.
	Suppose it is not; then by compactness there are finite formulas taken out of it whose conjunction is not consistent  
	(since $\Delta(M)$ 
	is closed under finite conjunctions, we can assume that just one formula is taken out of it). 
	Thus there are finitely many basic pairs $(\ua_1,\varphi_1(\uy_1)), \dots, (\ua_m,\varphi_m(\uy_m)) $, finitely many tuples of basic pairs 
	\begin{equation}\label{eq:basic}
		(\ua_{j1},\varphi_{j1}(\uy_{j1})), \dots, (\ua_{jm_j},\varphi_{jm_j}(\uy_{jm_j})) 
	\end{equation}
	($j=1,\dots, l$), finitely many tuples
	$\uc,\uc_1, \dots, \uc_l\in M$ such that
	$\cE_T$  proves
	\begin{equation}\label{eq:dd}
		C(\uc) \wedge 
		\bigwedge_{j=1}^l B_j(\uc_j,\underline{\kappa}^{\ua_{j1}},
		\dots,\underline{\kappa}^{\ua_{jm_j}}) 
		\wedge \bigwedge_{i=1}^m \varphi_i(\underline{\kappa}^{\ua_i}) \to \bot 
	\end{equation}
	where we have $M\models (\Diamond \varphi_i)(\ua_i)$ (for $i=1,\dots, m$) and $M\models (\Box B_j)(\uc_j,\ua_{j1},
	\dots,\ua_{jm_j})$ for all $j$.
	Notice that the basic pairs $(\ua_1,\varphi_1(\uy_1)), \dots, (\ua_m,\varphi_m(\uy_m)) $ are distinct from each other, but they might occur among  the ones in~\eqref{eq:basic} and the ones in~\eqref{eq:basic} may contain repetitions.
	
	Now all constants occurring in~\eqref{eq:dd} are fresh (in the sense that they do not occur in the language of $\cE_T$), so we can treat them as first order \emph{variables}.
	We can rewrite~\eqref{eq:dd} as
	\begin{equation*}
		\bigwedge_{j=1}^l B_j(\uc_j,\underline{\kappa}^{\ua_{j1}},
		\dots,\underline{\kappa}^{\ua_{jm_j}}) 
		\to \neg C(\uc) \vee \bigvee_{i=1}^m \neg\varphi_i(\underline{\kappa}^{\ua_i})
	\end{equation*}
	and apply the $\Box$ operator thus getting:
	\begin{equation}\label{eq:Fdd1}
		\bigwedge_{j=1}^l (\Box B_j)(\uc_j,\underline{\kappa}^{\ua_{j1}},
		\dots,\underline{\kappa}^{\ua_{jm_j}}) 
		\to \neg C(\uc) \vee \bigvee_{i=1}^m (\Box\neg\varphi_i)(\underline{\kappa}^{\ua_i})
	\end{equation}
	Here we used the product independence axiom to distribute the $\Box$ over disjunctions of formulas not sharing common variables;
	we used also the instance of the reflexivity axiom $\Box \neg C \to \neg C$ and the continuity axiom to permute the $\Box$ with the tuples of variables in the antecedent of the implication (such tuples of variables may contain repetitions, so they represent a proper substitution).
	Using classical tautologies, we can rewrite~\eqref{eq:Fdd1} as
	\begin{equation}\label{eq:Fdd2}
		\bigwedge_{j=1}^l (\Box B_j)(\uc_j,\underline{\kappa}^{\ua_{j1}},
		\dots,\underline{\kappa}^{\ua_{jm_j}}) 
		\wedge C(\uc) \wedge \bigwedge_{i=1}^m (\Diamond\varphi_i)(\underline{\kappa}^{\ua_i}) \to \bot
	\end{equation}
	This is in contrast to our data, because in $M$ we have
	$$
	M\models (\Box B_j)(\uc_j,{\ua_{j1}},\dots,{\ua_{jm_j}}), \quad M\models C(\uc), \quad  M\models
	(\Diamond\varphi_i)({\ua_i})
	$$
	for all $j=1, \dots l$ and $i=1, \dots, n$
	(recall that the $\uc, \uc_j,  \underline{\kappa}^{\ua_{jr}},\underline{\kappa}^{\ua_i}$ are variables in~\eqref{eq:Fdd1}, so that the $\cE_T$-provable formula~\eqref{eq:Fdd1} should be satisfied in $M$ by any tuple of elements from $M$, in particular by the tuple
	formed by the $\uc, \uc_j, {\ua_{jr}},{\ua_i}$).
	
	Now we know that $\Delta(M)\cup K_1\cup K_2$ is consistent, so it has a model $N$ which will be an elementary extension of $M$. The statement of the lemma asks for the existence of a set of relations $R=\{R_X\subseteq M(X)\times N(X)\}$ (indexed by the sorts of our language) such that: 
	\begin{description}
		\item[{\rm (i)}] whenever $(\ua,\varphi(\uy))$ is a basic pair there are $\ub$ such that $\ua R\ub$ componentwise holds and $N\models \varphi(\ub)$; 
		\item[{\rm (ii)}] whenever $\ua R\ub$ componentwise holds and $M\models(\Box B)(\ua)$ then $N\models B(\ub)$.
	\end{description}
	
	Notice that we do not have to care about `closure under products' of our family of relations $R$
	because, for product sorts, the relation is automatically taken to be the product of the components (in the second proof, we shall take another approach, requiring a direct definition of a family of relations closed under products). 
	
	The definition of our family of relations goes as follows: we say that $R_X$ includes all identical pairs 
	$(a,a)$ for $a\in M(X)$ together all pairs $(a_i,b_i)$  such that there is a basic pair
	$(\ua, \varphi)$  such that $a_i$ is the it-h component of the tuple $\ua$ and $b\in N(X)$ is such that $N\models b={\kappa}^{\ua}_i$ (in other words, $b$ is the interpretation in $N$ of the i-th component of the tuple of constants $\underline{\kappa}^{\ua}$).
	
	It is clear that (i) is satisfied because $N\models K_1$. To show (ii), we need the continuity axiom. Suppose we take a pair of tuples  $\uc=c_1, \dots, c_k$ and $\ud=:d_1,\dots, d_k$ such that $c_i R d_i$ holds for all $i=1,\dots,k$
	and such that
	$N\models (\Box B)(\uc)$. Then there are a natural number $l$, a tuple $\ua \in M(X)$, basic pairs 
	$(\ua_1,\varphi_1(\uy_1)), \dots, (\ua_m,\varphi_m(\uy_m)) $ 
	such that:
	\begin{itemize}
		\item $l$ is the sum of the lengths of the tuples $\ua,\ua_1, \dots, \ua_m$;
		\item $l_1, \dots l_k\leq l$;
		\item for all $j=1, \dots, k$, we have that $c_j$ is the $l_j$ element of the tuple $\ua,\ua_1, \dots, \ua_m$;
		\item for all $j=1, \dots, k$, we have that $d_j$ is the $l_j$ element of the tuple $\ua,
		\underline{\kappa}^{\ua_1},\dots,\underline{\kappa}^{\ua_m}$ (here we directly indicate with $\underline{\kappa}^{\ua_1},\dots,\underline{\kappa}^{\ua_m}$ the elements of $N$ interpreting the constants 
		$\underline{\kappa}^{\ua_1},\dots,\underline{\kappa}^{\ua_m}$).
	\end{itemize}
	The trick now is to replace the (atomic) formula $ (\Box B)(x_1,\dots, x_k)$
	with the (atomic) formula $ (\Box B)(x_{l_1},\dots, x_{l_k})$
	which is a formula
	containing at most the free variables $x_1,\dots, x_l$. 
	Then we have 
	that the tuple 
	$\ua,{\ua_1},\dots,{\ua_m}$ satisfies the formula 
	$(\Box B)(x_{l_1},\dots, x_{l_k})$ in $N$,
	hence also the formula 
	$\Box (B(x_{l_1},\dots, x_{l_k}))$
	by the continuity axiom; since $N\models K_2$, the latter  implies that the tuple $\ua,
	\underline{\kappa}^{\ua_1},\dots,\underline{\kappa}^{\ua_m}$
	satisfies the formula $B(x_{l_1},\dots, x_{l_k})$ in $N$, which means that $N\models B(\ud)$ holds.
\end{proof}

The second proof of Lemma~\ref{lem:extend-model} handles the above combinatorics in a conceptual way.
Before attacking it, we recall how to represent Robinson diagrams in  categorical logic  (this is mostly folklore information). Given an f-Boolean category $\cE$ and an object $X$ in it, the slice category $\cE/X$ `freely adds to $\cE$ a global element of type $X$': this means that there is an equivalence of categories between functors  $F_X: \cE/X \longrightarrow \cat{F}$ and pairs given by a functor 
$F:\cE \longrightarrow \cat{F}$ and a global element
$\mathbf{1}\buildrel{c}\over\longrightarrow F(X)$
(here by `functors' we mean functors preserving the involved logical structure). The functor $F$ is obtained by taking the composition $\cE\longrightarrow \cE/X \buildrel{F_X}\over \longrightarrow \cat{F}$ 
(here $\cE\longrightarrow \cE/X$ is the functor given by the pullback along $X\longrightarrow \mathbf{1}$)
and the global element $\mathbf{1}\buildrel{c}\over\longrightarrow F(X)$ is obtained applying $F_X$ to the diagonal 
\[\begin{tikzcd}
	X \ar[rr," \Delta_X "] \ar[rd, "1_X"'] &[-1em]&[-1em] X\times X \ar[ld,"\pi_2"] \\
	&X&
\end{tikzcd}\]
(we say that $F_X$ \emph{classifies} $F$ and $c$). Whenever we have a commutative triangle 
\[\begin{tikzcd}
	\cE/Y \ar[rr," f^*"] \ar[rd, "N_Y"'] &[-1em]&[-1em] \cE/X \ar[ld,"N_X"] \\
	&\cat{F}&
\end{tikzcd}\]
where $f^*$ is the pullback functor along $X\buildrel{f}\over\longrightarrow Y$, it can be shown that, if $N_X$ classifies $(N,c)$, then $N_Y$ classifies  the pair given by $N$ and the global element obtained by the composition $\mathbf{1}\buildrel{c}\over\longrightarrow N(X)\buildrel{N(f)}\over \longrightarrow N(Y)$.

Given now a Boolean model
$M:\cE\longrightarrow \cSet$, we can first form the \emph{category of elements} $\int M$ of $M$. This is a co-filtered category having as objects the pairs $(x,X)$ where $X$ is an object of $\cE$ and $x\in M(X)$. An arrow
$f: (x,X) \longrightarrow (y,Y)$ in $\int M$ is an arrow $f:X \longrightarrow Y$ in $\cE$ such that $M(f)(x)=y$.
To this category $\int M$ we can associate a pseudofunctor $D_M$ from $(\int M)^{op}$ to the 2-category of f-Boolean categories mapping an object $(x,X)$ to the slice categopry $\cE/X$ and an arrow $f:(x,X)\longrightarrow (y,Y)$ to the pullback functor $\cE/Y\longrightarrow \cE/X$. The filtered pseudo-colimit $\cat{L}_M$ of this functor (as well, by abuse, the pseudofunctor $D_M$ itself) is called the \emph{diagram} of $M$.

Consider now a model $\cat{L}_M\longrightarrow \cSet$: from the universal property of pseudocolimits, this model determines  commutative triangles 
\[\begin{tikzcd}
	\cE/Y \ar[rr," f^* "] \ar[rd, "N_Y"'] &[-1em]&[-1em] \cE/X \ar[ld,"N_X"] \\
	&\cSet&
\end{tikzcd}\]
varying $(a,X) \buildrel{f}\over{\longrightarrow} (b,Y)$ in $\int M$ (all this is up to unspecified suitably commuting iso's). Putting this together with the above information on slice categories, it follows that 
$\cat{L}_M\longrightarrow \cSet$ classifies a model $N:\cE\longrightarrow \cSet$ together with, for every $(a,X) \in \int M$,  elements $\mu_X(a)\in N(X)$. Such elements are such that, whenever we have an arrow $f:X\longrightarrow Y$ in $E$ and an element $b\in M(Y)$ such that $M(f)(a)=b$, we have also that 
$\mu_Y(b)=N(f)(\mu_X(a))$.
This is nothing but a natural transformation $\mu:M\longrightarrow N$, so that we can say that  \emph{the diagram $\cat{L}_M$  classifies elementary extensions of $M$}. This is precisely what we need for our proof.

\begin{proof}[Second Proof of Lemma~\ref{lem:extend-model}]

	Let $K = \setst{(Y,\varphi,y)}{Y \in \cE, \varphi \subseteq Y, y \in M(Y), y \vDash \Diamond\varphi}$: these are the basic pairs of the first proof of the lemma (now they became basic triples, because we display also the sort $Y$).
	An element $i \in K$ is thus decomposed as  $i= (Y,\varphi,y) \eqqcolon (\dom(i),\varphi(i),\pt(i)) $. Consider now the
	category $(\cSet_{fin}/K)^{op}$ whose objects
	are $K$-indexed finite families (below we use $\underline{n}, \underline{m},\dots$ to denote the finite sets $\{1, \dots, n\}, \{1,\dots,m\},\dots$).
	We extend the above notation to tuples 
	$$\sigma=\langle \sigma_i\rangle_{i\leq n}:\underline{n}\longrightarrow  K
	$$ 
	as follows: $\dom(\sigma) = \prod_{i\in \underline{n} } \dom(\sigma_i)$, $\varphi(\sigma) = \bigwedge_{i\in \underline{n}} \pi_{\dom(i)}^{-1} \varphi(\sigma_i)$ and $\pt(\sigma) = (\pt(\sigma_i)\,|\,i\leq n) \in M(\dom(\sigma))$.

	We now define a co-filtered category $I$ and a functor $\int M\longrightarrow I$.
	$I$ has objects the quadruples
	$$
	\langle x, X, \sigma, B\rangle
	$$
	where $(x,X)$ is an object of $\int M$ (thus $x\in M(X)$), $\sigma$ is an object of $(\cSet_{fin}/K)^{op}$ (thus it is a tuple from $K$) and $B \subseteq X\times \dom(\sigma)$ is such that
	\begin{equation}\label{eq:i1}
		(x,\pt \sigma) \vDash \Box B~.
	\end{equation}
	The arrow in $I$ between objects $\langle x, X, \sigma, B\rangle$ and $\langle x',X', \sigma', B'\rangle$ 
	are the pairs $(f,\alpha)$ such that $f: (x,X)\longrightarrow (x',X') \in\int N$
	and $\alpha: \sigma \longrightarrow \sigma'\in (\cSet_{fin}/K)^{op}$ satisfy
	the condition 
	\begin{equation}\label{eq:i2}
		B \leq (f\times \pi_{\alpha})^{-1}  B'
	\end{equation}
	where $\pi_{\alpha}$ is the tuple of projections 
	$\langle \pi_{\alpha(j)}\rangle_{j\leq \underline{n'}}$ (here 
	$\underline{n'}$ is the domain of $\sigma'$). We notice, \emph{en passant}, that the fact that $\alpha: \sigma \longrightarrow \sigma'$ is an arrow in $(\cSet_{fin}/K)^{op}$ implies that 
	\begin{equation}\label{eq:ipi}
		M(\pi_{\alpha})(\pt \sigma)= \pt \sigma'~~{\rm and}~~\pi_{\alpha}^{-1}(\varphi(\sigma'))= \varphi(\sigma)~~~.
	\end{equation}

	The category $I$ is co-filtered:  suppose we are given
	$$
	(f_1, \alpha_1),~(f_2,\alpha_2): \langle x_1, X_1, \sigma_1, B_1\rangle\rightrightarrows
	\langle x_2,X_2, \sigma_2, B_2\rangle
	$$
	and we want to get $(f, \alpha): \langle x', X', \sigma', B'\rangle\longrightarrow
	\langle x_1,X_1, \sigma_1, B_1\rangle$ equalizing these two arrows.
	Since $\int M$ is cofiltered, we can easily get $f$ equalizing $f_1,f_2$; to get $\alpha$, we consider the  equalizer of $\alpha_1, \alpha_2$ in $(\cSet_{fin})^{op}$. The tuple $\sigma'$ is supplied by the universal property of equalizers;  we finally put $B:= (f\times \pi_\alpha)^{-1}B_1$.
	Condition~\eqref{eq:i2} is trivial and condition~\eqref{eq:i1} follows from continuity and from the fact that $\alpha': \sigma'\longrightarrow \sigma_1$ is an arrow of  $(\cSet_{fin}/K)^{op}$, as noticed above in~\eqref{eq:ipi}.  
	The remaining cofiltering conditions for $I$ are checked in the same way.

	We have an obvious functor $\int M\longrightarrow I$ mapping $(x,X)$ to $(x,X, \emptyset, \top)$.
	We now need to define a pseudofunctor $D$ from $I^{op}$ to the 2-category of f-Boolean categories in such a way that the composition of $(\int M)^{op}\longrightarrow I^{op}$ with $D$ is precisely the diagram $D_M$ of $M$.
	Let $\langle x, X, \sigma, B\rangle$ be an object of $I$: recall that we have $B\wedge \pi^{-1}_{\dom(\sigma)}\varphi(\sigma)\hookrightarrow X\times \dom(\sigma)$. We stipulate that $D$ associate the slice category ${\cE}/(B\wedge \pi^{-1}_{\dom(\sigma)}\varphi(\sigma))$ with this object. Now take an arrow
	$(f, \alpha): \langle x_1, X_1, \sigma_1, B_1\rangle\longrightarrow
	\langle x_2,X_2, \sigma_2, B_2\rangle$
	and notice that from~\eqref{eq:i2} and~\eqref{eq:ipi} we can  obtain
	$$
	B \wedge \pi_{\dom(\sigma)}^{-1}(\varphi(\sigma)) \leq  (f\times \pi_{\alpha})^{-1}(B'\wedge \pi_{\dom(\sigma')}^{-1}(\varphi(\sigma'))
	$$
	This is equivalent to the fact that there is a (necessary unique) arrow making the square below commute:
	
	\begin{equation}\label{diag}
		\begin{tikzcd}
			B\land \pi^{-1}_{\dom(\sigma)}\varphi(\sigma) \ar[r," "] \ar[d," ",>->] & B'\land \pi^{-1}_{\dom(\sigma')}\varphi(\sigma') \ar[d," ",>->] \\
			X\times \dom(\sigma)\ar[r,"f\times \pi_{\alpha}"] & X'\times \dom(\sigma')
		\end{tikzcd}
	\end{equation}
	The pullback along this arrow will be the functor associated by $D$ to the arrow $(f,\alpha)$. This is clearly pseudo-functorial and, once composed with $(\int M)^{op}\longrightarrow I^{op}$, gives the diagram $D_M$ of the model $M$.
	
	It remains to check that the pseudocolimit $\cat{L}$ of the pseudofunctor $D$ is consistent (an f-Boolean category is consistent iff the Boolean algebra of the subobjects of the terminal object is not degenerate). Since the colimit is filtered, this amount to show that the f-Boolean category associated with each object $\langle x, X, \sigma, B\rangle$  of $I$ is consistent. Thus we need to check that for every object $\langle x, X, \sigma, B\rangle$ we do not have
	$B \land \pi_{\dom(\sigma)}^{-1}\varphi(\sigma) = \bot$. If this is the case (by absurd), then we have also
	$\bot = \Box B \land \Diamond \pi_{\dom(\sigma)}^{-1} \varphi(\sigma)$; by the product independence axiom, we would get
	$$
	\bot= \Box B \land \bigwedge_{i}\pi_{\dom(\sigma_i)}^{-1}\Diamond \varphi(\sigma_i)
	$$
	which is in contrast to~\eqref{eq:i1} and $\pt(\sigma_i)\vDash \Diamond \varphi(\sigma_i)$ (the latter comes from the fact that $\sigma_i\in K$).

	Since $\cat{L}$  is consistent, by G\"odel's completeness theorem there is a Boolean model ${\cat L}\longrightarrow
	\cSet$ whose restriction (induced by the composition with $(\int M)^{op}\longrightarrow I^{op}$)
	$$
	N:\cat{L}_M \longrightarrow {\cat L}\longrightarrow
	\cSet
	$$
	to the diagram $\cat{L}_M$ of $M$ classifies an extension  of $M$. This extension, call it $\mu:M\longrightarrow N$,  has the following properties by construction:
	\begin{itemize}
		\item for every $i\in K$ there is $\kappa_i\in N(\dom(i))$
		such that $\kappa_i\models^N \varphi(i)$ (as usual, we extend the notation 
		to the $\sigma \in (\cSet_{fin}/K)^{op}$ and let
		$\kappa(\sigma)$ be the tuple $\kappa(\sigma) = (\kappa_i\,|\,i\leq n) \in N(\dom(\sigma))$);
		\item for every $a\in M(X)$, for every $\sigma \in (\cSet_{fin}/K)^{op}$
		and $B\subseteq X\times \dom(\sigma)$, if 
		$(a, \pt \sigma) \models^M \Box B$ then $(\mu(a), \kappa(\sigma)) \models^N B$.
	\end{itemize}
	
	Armed with these data, we can define the desired product closed family of relations
	$R=\{R_Z \subseteq M(Z) \times N(Z)\}_Z$ as follows. We have that $(c,d)\in M(Z)\times N(Z)$  belongs to $R_Z$ iff there are $X, \sigma, a\in M(X)$
	such that $Z\simeq X\times \dom(\sigma)$,
	$c\simeq (a, \pt(\sigma))$ and $d\simeq (\mu_X(a), \kappa(\sigma))$. 
\end{proof}

\subsubsection{The proof of the main result}

We now continue our construction of topological models.

\begin{lemma}\label{lem:closure-relations}
	Let $M_1, M_2, \cdots$ be Boolean models of $\cE$. For each $i\geq 1$, let $R_i \subseteq M_i \times M_{i+1}$ be a relation stable under products and such that $R_i^{-1} \llbracket\varphi\rrbracket_{M_{i+1}} \subseteq \llbracket\Diamond\varphi\rrbracket_{M_i}$ for all $\varphi \subseteq X$ in $\cE$. Then there is a family of relations $(\ovl{R}_{ij} \subseteq M_i \times M_j)_{i<j}$ which satisfies:
	\begin{enumerate}[label=(L\arabic*)]
		\item $R_i \subseteq \ovl{R}_{i,i+1}$,\label{clR:1}
		\item $\ovl{R}_{ij} \ovl{R}_{jk} \subseteq \ovl{R}_{ik}$,\label{clR:2}
		\item $\ovl{R}_{ij}$ is stable under partial maps,\label{clR:3}
		\item $\ovl{R}_{ij}$ is stable under products,\label{clR:4}
		\item $\ovl{R}_{ij}^{-1} \llbracket\varphi\rrbracket_{M_j} \subseteq \llbracket\Diamond\varphi\rrbracket_{M_i}$ for all $\varphi \subseteq X$ in $\cE$.\label{clR:5}
	\end{enumerate}
\end{lemma}

\begin{proof}
	We simply define $(\ovl{R}_{ij})_{i<j}$ as the closure of the $R_i$ under the rules \ref{clR:2}, \ref{clR:3}, \ref{clR:4}. We will use the loop contraction axiom \eqref{ax:LC} to show that \ref{clR:5} holds.
	
	Let $X \in \cE$, $a \in M_i(X)$ and $b \in M_j(X)$ such that $a \ovl{R}_{ij} b$. To each such pair $(a,b)$, we will associate a finite path whose nodes and arrows are labeled according to the rules below:
	\begin{itemize}
		\item Each node is labeled by a triple $(k,Y,y)$ where $i \leq k \leq j$, $Y \in \cE$ and $y \in M_k(Y)$.
		\item The first node is labeled by $(i,X,a)$ and the last node is labeled by $(j,X,b)$.
		\item There are two types of arrow:
		\begin{itemize}
			\item First type: $(k,Y,y) \longrightarrow (k+1,Y,y')$ with $y R_k y'$, labeled by $1_Y$.
			\item Second type: $(k,Y,y) \longrightarrow (k,Y',y')$ labeled by an $S \subseteq Y'\times Y$ such that $y' S y$ (here we simply write $S$ instead of $M_k(S)$).
		\end{itemize}
		\item The sequence of labels of the arrows forms an acceptable loop.
	\end{itemize}
	To define this path, we proceed by induction on the way that $(a,b) \in \ovl{R}_{ij}$ is constructed inductively following the rules \ref{clR:1}--\ref{clR:4}.
	\begin{itemize}
		\item For the base case \ref{clR:1}, we have $(a,b) \in \ovl{R}_i(X)$ for some $X \in \cE$ and the path simply consists of one arrow of the first type $(i,X,a) \longrightarrow (i+1,X,b)$. The associated loop is acceptable since it consists of just an identity.
		\item For \ref{clR:2}, we have $(a,b) \in \ovl{R}_{ij}(X)$ and $(b,c) \in \ovl{R}_{jk}(X)$. The last node of the path of $(a,b)$ has the same label as the first node of the path of $(b,c)$. We can thus concatenate these two paths by merging these nodes. The acceptable loops are stable by concatenation by \ref{acc:3}, hence the path is valid.
		\item For \ref{clR:3}, suppose that we have already associated a path to $(a,b) \in \ovl{R}_{ij}(X)$. Let $f : X\relto Y$ be a partial map such that $a$ and $b$ are both in the domain of $f$. We associate to $(f(a),f(b)) \in \ovl{R}_{ij}(Y)$ the path below, where the dashed arrow represents the path of $(a,b)$ and where the solid arrows represent new nodes of the second type.
		\[\begin{tikzcd}
			(i,Y,f(a)) \ar[r,"f"] & (i,X,a) \ar[r,dashed] & (j,X,b) \ar[r,"f^{-1}"] & (j,Y,f(b))
		\end{tikzcd}\]
		Thanks to \ref{acc:4}, the underlying loop is again acceptable.
		\item For \ref{clR:4}, suppose that we have associated a path to both $(a,b) \in \ovl{R}_{ij}(X)$ and $(a',b') \in \ovl{R}_{ij}(X)$. First of all, notice that we can duplicate a node in a path by adding an arrow of the second type labeled by the identity between the two copies, thanks to \ref{acc:2}. By adding such identities, it is possible to modify the paths of $(a,b)$ and $(a',b')$ so that the arrows of the first type become aligned (i.e., the two paths have the same lengths and the arrows with matching positions have the same type). Once this is done, we take the ``pointwise product'' of the two paths: Two nodes $(k,Y,y)$ and $(k,Y',y')$ produce the node $(k,Y\times Y',(y,y'))$; Two arrows of the first type produce an arrow of the first type, using the stability of the $R_i$ under products; Two arrows of the second type labeled by $S$ and $S'$ produce an arrow of the second type labeled by $S\times S'$. Thanks to the stability of acceptable loops under products asserted by \ref{acc:5}, the path obtained is valid.
	\end{itemize}

	In the last step of the proof, we use the path associated to $(a,b) \in \ovl{R}_{ij}(X)$ to show that if $b \vDash \varphi$ then $a \vDash \Diamond\varphi$ for all $\varphi \subseteq X$. This implies that $\ovl{R}_{ij}^{-1} \llbracket\varphi\rrbracket_{M_j} \subseteq \llbracket\Diamond\varphi\rrbracket_{M_i}$. Suppose that $b \vDash \varphi$. Given a node $(k,Y,y)$ in the path, we denote by $\Ss(k,Y,y)$ the operator 
    $\Diamond S_u \Diamond S_{u+1} \Diamond \cdots \Diamond S_v $
    where $S_u,{\ldots},S_v$ is the sequence of labels of the arrows found after $(k,Y,y)$. (Even though it is possible that two nodes have the same label, we denote a node by its label as it creates no confusion in the proof.) We will show that
	\[ y \vDash \Ss(k,Y,y) \varphi \]
	for every node $(k,Y,y)$. The proof is by induction on the position of the node, starting from the end.
	\begin{itemize}
		\item The last node is $(j,X,b)$ and we indeed have $b \vDash \varphi$ by assumption.
		\item Suppose that we encounter an arrow of the first type $(k,Y,y) \longrightarrow (k+1,Y,y')$ and we already know that $y' \vDash \Ss(k+1,Y,y') \varphi$. Then $y R_k y'$, so that $y \vDash \Diamond \Ss(k+1,Y,y') \varphi$ by the assumption on $R_k$. Since 
        $\Ss(k,Y,y) = \Diamond\Ss(k+1,Y,y')$, 
        we get $y \vDash \Ss(k,Y,y) \varphi$.
		\item Suppose that we encounter an arrow of the second type $(k,Y,y) \longrightarrow (k,Y',y')$ labeled by $S$. Suppose that $y' \vDash \Ss(k,Y',y') \varphi$. Since $y' S y$, we get $y \vDash S \Ss(k,Y',y') \varphi \leq \Diamond S \Ss(k,Y',y') \varphi$. Noting that $\Ss(k,Y,y) = \Diamond S \Ss(k,Y',y')$, we obtain $y \vDash \Ss(k,Y,y) \varphi$.
	\end{itemize}
	At the end of the induction, we get 
    $a \vDash \Ss(i,X,a) \varphi$.
    Given that the 
    whole sequence 
    forms 
    an acceptable loop, 
    we deduce from \eqref{ax:LC} that $a \vDash \Diamond \varphi$.
\end{proof}

We are finally ready to prove the main theorem.

\begin{theorem*}[\ref{thm:completeness-top}]
	A small modal category is topological if and only if it admits a conservative modal functor to a power of $\cTop$.
\end{theorem*}

\begin{proof}
	The right-to-left implication is Proposition~\ref{prop:top-soundness}. It remains to show that if $\cE$ is a small topological modal category, then the modal functors $\cE \to \cTop$ are jointly conservative.
	
	By a standard argument, since we are working with negations, it is enough to show that for each $U \neq \bot$, there is a topological model $M : \cE \to \cTop$ with $M(U) \neq \emptyset$. By Gödel's completeness theorem, we know that there is a Boolean model $M_0 : \cE \to \cSet$ with $M_0(U) \neq \emptyset$. We use Lemma~\ref{lem:extend-model} repeatedly to build a sequence $M_0 \to M_1 \to \cdots$ of Boolean models with relations $R_i \subseteq M_i\times M_{i+1}$. By Lemma~\ref{lem:closure-relations}, we can extend these relations to a family $(\ovl{R}_{ij} \subseteq M_i \times M_j)_{i<j}$ satisfying \ref{clR:1}--\ref{clR:5}.
	
	From this, we build a functor $\cE \to \cSetRel^\omega$. Each object $X \in \cE$ is sent to the sequence $M_0(X) \to M_1(X) \to \cdots$ with relations $\ovl{R}_{ij}(X) \subseteq M_i(X)\times M_j(X)$. This is a relational sequence: by \ref{clR:1}, each $\ovl{R}_{ij}(X)$ contains the graph of $M_i(X)\to M_j(X)$ and by \ref{clR:2}, $\ovl{R}_{ij} \ovl{R}_{jk} \subseteq \ovl{R}_{ik}$. Moreover, by \ref{clR:3} the relations $\ovl{R}_{ij}$ are stable under the morphisms of $\cE$ and this defines a functor $\cE \to \cSetRel^\omega$. Let us show that this functor preserves embeddings and finite limits.
	
	\proofstep{Embeddings} Let $f : Y\hookrightarrow X$ be an embedding in $\cE$. The inverse $f^{-1} : X\relto Y$ is a partial map and since $\ovl{R}_{ij}$ is stable under partial maps by \ref{clR:3}, $\ovl{R}_{ij}(Y)$ is the restriction of $\ovl{R}_{ij}(X) \subseteq M_i(X)\times M_j(X)$ to $M_i(Y)\times M_j(Y) \subseteq M_i(X)\times M_j(X)$. Thus $\cE \to \cSetRel^\omega$ preserves embeddings.
	
	\proofstep{Finite products} The $\ovl{R}_{ij}$ are stable under products by \ref{clR:4} and under the actions of morphisms by \ref{clR:3}, so that $\ovl{R}_{ij}(X\times Y) = \ovl{R}_{ij}(X) \times \ovl{R}_{ij}(Y)$. Thus $\cE \to \cSetRel^\omega$ preserves finite products.
	
	\proofstep{Equalizers} Since equalizers in $\cE$ are embeddings, they are sent to embeddings by $\cE \to \cSetRel^\omega$. Using that each $M_i({-})$ in the sequence preserves equalizers, and since the equalizers in $\cSetRel^\omega$ are the pointwise equalizers that are also embeddings, the functor $\cE \to \cSetRel^\omega$ preserves equalizers.
	
	Now, we compose $\cE \to \cSetRel^\omega$ with the functor $c : \cSetRel^\omega \to \cTop$ of \S~\ref{subsec:top-mod}. We obtain a functor $M : \cE \to \cTop$ which preserves embeddings and finite limits by Lemmas~\ref{lem:c-embed} and \ref{lem:c-is-lex}. It also preserves surjections and joins of subobjects: If $f : X\twoheadrightarrow Y$ is a surjection in $\cE$, then $M_i(X)\twoheadrightarrow M_i(Y)$ too for all $i$, hence the colimit is a surjection; If $X = \varphi\lor\psi$ in $\cE$, then similarly $M_i(X) = M_i(\varphi)\cup M_i(\psi)$ and this is also transferred to the colimit.
	
	It remains to show that $M : \cE \to \cTop$ preserves the modalities. Let $\varphi \subseteq X$ be a subobject in $\cE$. We wish to show that $M(\Diamond\varphi) = \ovl{M(\varphi)}$ in $M(X)$.
	
	We start with the inclusion $M(\Diamond\varphi) \subseteq \ovl{M(\varphi)}$. Let $[p] \in M(\Diamond\varphi) \subseteq M(X)$. In order to get $[p] \in \ovl{M(\varphi)}$, we need to show that $I_q \cap M(\varphi) \neq \emptyset$ for any $q \in M_i(\Diamond\varphi)$ with $[p] = [q]$. We have $q \in M_i(\Diamond\varphi) = R_i^{-1} M_{i+1}(\varphi) \subseteq \ovl{R}_{i,i+1}^{-1} M_{i+1}(\varphi)$. Hence there is $q' \in M_{i+1}(\varphi)$ such that $q \ovl{R} q'$, and $[q'] \in I_q \cap M(\varphi)$ as desired.
	
	Second, we prove that $\ovl{M(\varphi)} \subseteq M(\Diamond\varphi)$. Let $[p] \in \ovl{M(\varphi)}$ with $p \in M_i(X)$. In particular, $I_p \cap M(\varphi)$ contains some point $[p']$ with $p' \in M_j(\varphi)$ and $p \ovl{R}_{ij} p'$. Since $\ovl{R}_{ij}^{-1} M_j(\varphi) \subseteq M_i(\Diamond\varphi)$, we get $[p] \in M(\Diamond\varphi)$ as claimed.
	
	To conclude, we have shown that $M : \cE\to\cTop$ is a topological model. Moreover, since $M_0(U)\neq\emptyset$, the colimit $M(U) = \colim_i M_i(U)$ is nonempty.
\end{proof}

\section{Syntactic topological modal categories}
\label{sec:syntax-doctrines}

The purpose of this section is to explain how topological modal categories can be built concretely or, alternatively, how the axioms of topological modal categories translate in terms of a first-order calculus. Given a first-order modal theory $\TT$, we build its syntactic topological modal category $\Synt_\cSFour(\TT)$, whose category of topological models is equivalent to that of $\TT$. In order to formalize first-order modal theories, we use \emph{first-order doctrines}. Building a first-order doctrine out of a
syntactic calculus for a theory in a first order language is a standard construction (one takes terms as arrows in the base category $\cC$ and the functor $D$ from Definition~\ref{def:Booleandoctrine} below is obtained by taking equivalence classes of formulae under provable equivalence). In turn,
syntactic calculi for first order modal theories (with substitutions semi-commuting with modal operators) can be built for instance according to the guidelines of~\cite{GM2025} or of~\cite{handbook}. 

\subsection{Boolean doctrines}

We start by recalling the notion of a Boolean doctrine on a small cartesian category $\cC$. We refer the reader to \cite{PitCategoricalLogic2001} for more details.

Let \eqref{diag:BC} be a commutative square in the category of Boolean algebras and suppose that $f^*$ and $g^*$ have respective left adjoints $f_*$ and $g_*$. We say that \eqref{diag:BC} is a \emph{Beck--Chevalley} square if \eqref{diag:BC-} commutes.\\
\begin{minipage}{0.35\textwidth}
\begin{equation}\label{diag:BC}\begin{tikzcd}
		X \ar[r,"f^*"] \ar[d,"u^*"'] & Y \ar[d,"v^*"] \\
		Z \ar[r,"g^*"] & W
\end{tikzcd}\end{equation}
\end{minipage}\hfill\begin{minipage}{0.35\textwidth}
\begin{equation}\label{diag:BC-}\begin{tikzcd}
		X \ar[d,"u^*"'] & Y \ar[d,"v^*"] \ar[l,"f_*"'] \\
		Z & X \ar[l,"g_*"']
\end{tikzcd}\end{equation}
\end{minipage}

\begin{definition}\label{def:Booleandoctrine}
	A \emph{Boolean doctrine} on $\cC$ is a functor $D : \cC^\op \to \cBoolAlg$ such that:
	\begin{enumerate}
		\item Each $D(f)$ has a left adjoint $\exists_f$.
		\item Frobenius reciprocity holds: $\exists_f (D(f)(\varphi) \land \psi) = \varphi \land \exists_f \psi$.
		\item Every square in $\cC$ of the form \eqref{diag:BC-exists} or \eqref{diag:BC-equals} is sent by $D$ to a Beck--Chevalley square, where $\Delta_X$ denotes the diagonal $X\to X\times X$.\\
		\begin{minipage}{0.35\textwidth}
			\begin{equation}\label{diag:BC-exists}\begin{tikzcd}
					X\times Y \ar[r,"\pi_Y"] \ar[d,"X\times f"'] & Y \ar[d,"f"] \\
					X\times Z \ar[r,"\pi_Z"] & Z
			\end{tikzcd}\end{equation}
		\end{minipage}\hfill\begin{minipage}{0.4\textwidth}
			\begin{equation}\label{diag:BC-equals}\begin{tikzcd}
					X\times Y \ar[d,"X\times f"'] \ar[r,"\Delta_X\times Y"] & X\times X\times Y \ar[d,"X\times X\times f"] \\
					X\times Z \ar[r,"\Delta_X\times Z"] & X\times X\times Z
			\end{tikzcd}\end{equation}
		\end{minipage}
	\end{enumerate}
\end{definition}

\begin{example}
	The powerset functor $\Ps : \cSet^\op \to \cBoolAlg$ is a Boolean doctrine. More generally, for any f-Boolean category $\cE$, the functor $\Sub_\cE : \cE^\op \to \cBoolAlg$ is a Boolean doctrine \cite{GM2025}.
\end{example}

Boolean doctrines algebraize Boolean first-order logic with equality, whose syntax can thus be used to manipulate the elements of a Boolean doctrine. In fact, the reader unfamiliar with doctrines can simply consider that the elements of a Boolean doctrine $D$ are the formulas in a fixed first-order language modulo equivalence in a fixed theory, with the order of entailment. In the table below, we present three equivalent notation styles to manipulate the elements of a Boolean doctrine. We will mainly use the set-theoretical and the logical notations.

\begin{center}\begin{tabular}{c | c | c}
	\toprule
	Algebraic notation & Set-theoretical notation & Logical notation \\
	\midrule
	$\varphi \in D(X\times Y)$ & $\varphi \subseteq X\times Y$ & $\varphi(x,y)$ in context $x : X, y : Y$\\
	$\varphi \leq \psi$ & $\varphi \subseteq \psi$ & $\varphi(x) \vdash \psi(x)$\\
	$D(f)(\varphi)$ & $f^{-1}(\varphi)$ & $\varphi(f(x))$ \\
	$\exists_f(\varphi)$ & $f[\varphi]$ & $\exists x : f(x)=y \land \varphi(x)$\\
	$D(\Delta_X)(\varphi)$ & $\Delta_X^{-1}(\varphi)$ & $\varphi(x,x)$\\
	$\exists_{\pi_X}(\varphi)$ & $\pi_X[\varphi]$ & $\exists x : \varphi(x,y)$\\
	$\exists_{\Delta_X}(\top)$ & $\Delta_X[\top]$ & $x = y$\\
	\bottomrule
\end{tabular}\end{center}

\paragraph{Models} Let $D : \cC^\op \to \cBoolAlg$ be a Boolean doctrine and let $\cE$ be a Boolean category. A \emph{model} of $D$ in $\cE$ is a cartesian functor $F : \cC \to \cE$ equipped with a natural transformation $\alpha_X : D(X) \to \Sub_\cE(F(X))$ whose naturality squares are Beck--Chevalley. We write $\llbracket\varphi\rrbracket_F$ instead of $\alpha_X(\varphi)$ when $\varphi \in D(X)$, and we leave $\alpha$ implicit. A \emph{morphism} from a model $F$ to a model $G$ is a natural transformation $t : F\to G$ such that $t_X(\llbracket\varphi\rrbracket_F) \subseteq \llbracket\varphi\rrbracket_G$ for all $\varphi \in D(X)$. (In fact, due to the presence of complements, we even have $t_X(\llbracket\varphi\rrbracket_F) = \llbracket\varphi\rrbracket_G$.)

\paragraph{Partial maps} Given a Boolean doctrine $D : \cC^\op \to \cBoolAlg$, a \emph{partial map} from $X \in \cC$ to $Y \in \cC$ (relatively to $D$) is a relation $R \subseteq X\times Y$ internally satisfying the usual axiom $R(x,y) \land R(x,y') \vdash y=y'$. The partial maps from $X$ to $Y$ will be written as $f : X\relto Y$. The corresponding relation is written $f(x)=y$, but beware that it is only a notation: we are not using the equality predicate here. In some cases, it makes sense to use $f(x)$ as a term, for instance if $f$ is total. The domain $\dom(f)$ of $f$ is $[\exists y:f(x)=y] \in D(X)$. The image $\im(f)$ of $f$ is $[\exists x:f(x)=y] \in D(Y)$. Partial maps can be composed. A partial map $f : X\relto Y$ can be \emph{restricted} to a subdomain $\varphi \subseteq \dom(f)$ by defining $\restr{f}{\varphi}(x)=y \iff \varphi(x) \land (f(x)=y)$.

\paragraph{The syntactic Boolean category} Every Boolean doctrine $D : \cC^\op \to \cBoolAlg$ has a canonically associated Boolean category, called its \emph{syntactic category} and written $\Synt_\cBA(D)$. The morphisms from $\Synt_\cBA(D)$ to another Boolean category $\cE$ are equivalent to the models of $D$ in $\cE$. We assume known the construction of $\Synt_\cBA(D)$. Its objects are the elements $\varphi \in D(X)$ of $D$. The morphisms $\varphi\to\psi$, which we call the \emph{maps} from $\varphi$ to $\psi$, are the partial maps $f : X\relto Y$ with $\dom(f)=\varphi$ and $\im(f) \subseteq \psi$. 

\begin{remark}
	Every morphism $f : X\to Y$ in $\cC$ defines a morphism in $\Synt_\cBA(D)$, via its graph. However, it is possible that two distinct parallel morphisms in $\cC$ are sent to the same morphism in $\Synt_\cBA(D)$. This simply means that the theory classified by $D$ proves $\vdash f(x)=g(x)$.
\end{remark}

\begin{remark}
The construction $\Synt_\cBA(D)$ can also be used to convert an f-Boolean category into a Boolean category: recall indeed that an f-Boolean category $\cE$ can be seen as the  Boolean doctrine $(\cC, D)$, where $\cC$ is $\cE$ and $D$ is the subobject functor $\Sub_\cE$.
In this case $\Synt_\cBA(D)$ has the same objects as $\cE$ (since $\cE$ satisfies comprehension), however more isomorphisms arise. Take for instance a functional relation $R\buildrel{r}\over{\hookrightarrow} X\times Y$;
the composite arrow $R\buildrel{r}\over{\hookrightarrow} X\times Y\buildrel{\pi_X}\over \longrightarrow X$ might not be an iso in $\cE$ (by definition of a ``functional relation,'' it is only an injective surjection, i.e., a monomorphism which is also in the left class of the factorization system), but it becomes an iso when $\cE$ is embedded into the Boolean category $\Synt_\cBA(D)$.
\end{remark}

\subsection{Topological modal doctrines}

An S4 modal algebra is a modal algebra satisfying the \eqref{ax:S4} axiom $S \leq \Diamond S = \Diamond\Diamond S$. Recall that a lax morphism of S4 algebras is a Boolean algebra morphism $f : A \to B$ such that $\Diamond f(x) \leq f(\Diamond x)$. We denote by $\cIntAlg$ the category of S4 modal algebras and lax morphisms.

\begin{definition}\label{dfn:top-mod-doc}
	A \emph{topological modal doctrine} on a small cartesian category $\cC$ is a functor $D : \cC^\op \to \cIntAlg$ such that:
	\begin{enumerate}
		\item The composite with the forgetful functor $\cIntAlg \to \cBoolAlg$ is a Boolean doctrine.
        \item The axiom \eqref{ax:PI} holds for any two objects in $\cC$.
		\item There is a distinguished set $\Cc(D)$ of partial maps of the underlying Boolean doctrine. They are closed under compositions, products, and they contain all the maps from $\cC$.
		\item Similarly to \S~\ref{subsec:loop-contract}, we define a loop as a composable sequence of partial maps in $\Cc(D)$ such that the codomain of the first map is equal to the domain of the last map. The acceptable loops are defined by \ref{acc:1}--\ref{acc:5}, except that the partial map appearing in \ref{acc:4} is taken in $\Cc(D)$. We require that \eqref{ax:LC} holds for all the acceptable loops.
	\end{enumerate}
\end{definition}

\begin{remark}
	The arrows in $\Cc(D)$ are intended to represent ``partial continuous maps.''
	In order to express that some injection $f : X \to Y$ is an embedding, we add to $\Cc(D)$ the partial inverse function $Y \relto X$, with domain $f[X] \subseteq Y$.
	Just like in Remark~\ref{rmq:loop-contraction-trivializes}, the axiom \eqref{ax:LC} can be removed if $\Cc(D)$ contains only the maps from $\cC$.
\end{remark}

\paragraph{Models} The notion of model of a Boolean doctrine extends to topological modal doctrines. Let $D : \cC^\op \to \cIntAlg$ be a topological modal doctrine and let $\cE$ be a topological modal category. A \emph{model} of $D$ in $\cE$ is a cartesian functor $F : \cC \to \cE$ equipped with a natural transformation $\alpha_X : D(X) \to \Sub_\cE(F(X))$ such that:
\begin{itemize}
	\item Each $\alpha_X$ is a morphism of modal algebras and not a mere lax morphism, i.e., $\alpha_X(\Diamond\varphi) = \Diamond\alpha(\varphi)$.
	\item The naturality squares of $\alpha$ are Beck--Chevalley.
	\item Each partial map in $\Cc(D)$ is sent to the graph of a partial map in $\cE$.
     More explicitly, if $R \in D(X\times Y)$ is in $\Cc(D)$, then $\alpha(R) \subseteq F(X)\times F(Y)$ is the graph of a partial map $F(X)\relto F(Y)$ in $\cE$ (partial maps in a modal category $\cE$ were introduced in Subsection~\ref{subsec:modalcategopries}). When $\cE$ is the category of topological spaces, this means that each partial map in $\Cc(D)$ is interpreted as a continuous partial map in the usual sense.
\end{itemize}
We write $\llbracket\varphi\rrbracket_F$ instead of $\alpha_X(\varphi)$ when $\varphi \in D(X)$, and we leave $\alpha$ implicit. A \emph{morphism} from a model $F$ to a model $G$ is a natural transformation $t : F\to G$ such that $t_X(\llbracket\varphi\rrbracket_F) \subseteq \llbracket\varphi\rrbracket_G$ for all $\varphi \in D(X)$.

\subsection{The syntactic topological modal category}

In the remainder of the paper, we build the syntactic category of a topological modal doctrine. We show that it is a topological modal category, but we leave it to the reader to prove its universal property relatively to models in topological modal categories.

Let $D : \cC^\op \to \cIntAlg$ be a topological modal doctrine. It can be seen as a bare Boolean doctrine and thus $\Synt_\cBA(D)$ is a Boolean category whose morphisms are called \emph{maps}. A map is called \emph{continuous} if it is the restriction of a partial map in $\Cc(D)$.

\begin{lemma}
	Continuous maps are stable under restrictions, compositions and products.
\end{lemma}

\begin{proof}
	The stability under restrictions is obvious. If $\restr{f}{\varphi} : \varphi\to\psi$ and $\restr{g}{\psi} : \psi\to\theta$ are two maps obtained as the restrictions of $f : X\relto Y$ and $g : Y\relto Z$, then $\restr{g}{\psi} \circ \restr{f}{\varphi} = \restr{(g\circ f)}{\varphi}$. This shows the stability under composition. Let $\restr{f}{\varphi} : \varphi\to\psi$ and $\restr{f'}{\varphi'} : \varphi'\to\psi'$ be two maps obtained as the restrictions of $f : X\relto Y$ and $g : X'\relto Y'$. Then $\restr{f}{\varphi} \times \restr{f'}{\varphi'} = \restr{(f\times f')}{\varphi\times\varphi'}$. This shows the stability under products.
\end{proof}

The continuous maps thus form a wide subcategory of $\Synt_\cBA(D)$. We denote it by $\Synt_\cSFour(D)$ and we call it the \emph{syntactic category} of $D$.

\begin{lemma}
	$\Synt_\cSFour(D)$ inherits finite limits from $\Synt_\cBA(D)$.
\end{lemma}

\begin{proof}
	Let $\varphi \subseteq X$ and $\psi \subseteq Y$ be two objects of $\Synt_\cSFour(D)$. Their product $\varphi\times\psi$ in $\Synt_\cBA(D)$ is $\varphi(x) \land \psi(y) \subseteq X\times Y$. The projections $\varphi\times\psi \to \varphi$ and $\varphi\times\psi\to\psi$ are the restrictions of the projections $X\times Y\to X$ and $X\times Y\to Y$, hence they are continuous. The diagonals $\theta\to\theta\times\theta$ are also continuous, since they are the restrictions of diagonals. Given any two continuous maps $f:\theta\to\varphi$ and $g:\theta\to\psi$, the map $(f,g) : \theta\to\varphi\times\psi$ is the composition of the diagonal $\theta\to\theta\times\theta$ and $f\times g:\theta\times\theta\to\varphi\times\psi$, hence it is continuous. This shows that $\varphi\times\psi$ is also the product of $\varphi$ and $\psi$ in $\Synt_\cSFour(D)$.
	
	Let $f, g : \varphi\rightrightarrows\psi$ be two parallel continuous maps. Their equalizer $\eq(f,g)$ in $\Synt_\cBA(D)$ is $\varphi(x) \land (f(x)=g(x))$. The canonical inclusion $\eq(f,g) \hookrightarrow \varphi$ is continuous, as it is the restriction of the identity. Given a continuous map $h : \theta\to\varphi$ such that $hf = hg$, we obtain that $\im(h) \subseteq \eq(f,g)$, hence $h$ is also a continuous map $\theta\to\eq(f,g)$. This shows that $\eq(f,g)$ is also the equalizer of $f$ and $g$ in $\Synt_\cSFour(D)$.
\end{proof}

We now define the factorization system on $\Synt_\cSFour(D)$. The left class $\Ec$ is the class of continuous surjections. In other words, $\Ec$ is class of continuous maps which are regular epimorphisms in $\Synt_\cBA(D)$. The right class $\Mc$ is the isomorphism closure in $\Synt_\cSFour(D)$ of the maps which are the restriction of an identity of an object of $\cC$.

\begin{lemma}
	$(\Ec,\Mc)$ is an orthogonal factorization system on $\Synt_\cSFour(D)$.
\end{lemma}

\begin{proof}
	Every continuous map $f : \varphi\to\psi$ can be factored as a surjection $\varphi\twoheadrightarrow\im(f)$ in $\Ec$ followed by the canonical inclusion $\im(f) \subseteq \psi$ in $\Mc$. Note that $\varphi\twoheadrightarrow\im(f)$ is continuous since it has the same graph as $f$. We need to check the orthogonality condition. Let $f : \varphi\twoheadrightarrow\psi$ be a surjection in $\Ec$, and let $i : \theta \hookrightarrow \xi$ be a continuous map in $\Mc$, which without loss of generality we suppose to be the restriction of an identity. Suppose we have a commutative square as below in $\Synt_\cSFour(D)$.
	\[\begin{tikzcd}
		\varphi \ar[r,"f",->>] \ar[d,"u"'] & \psi \ar[d,"v"]\\
		\theta \ar[r,hook,"i"] & \xi
	\end{tikzcd}\]
	By orthogonality in $\Synt_\cBA(D)$ there is a map $\psi\to\theta$ making the two triangles commute. This map is continuous since it is a restriction of $v$ which is continuous.
\end{proof}

\begin{lemma}
	$\Ec$ is stable under pullbacks and $\Mc$ contains all the regular monomorphisms.
\end{lemma}

\begin{proof}
	That $\Ec$ is stable under pullbacks follows directly from the fact that finite limits are inherited from $\Synt_{\cBA}(D)$ and that surjections are stable under pullbacks in $\Synt_{\cBA}(D)$. By the computation of equalizers in $\Synt_{\cBA}(D)$, we see that they are indeed restrictions of identities, hence in $\Mc$.
\end{proof}

\begin{lemma}
	For any $\varphi \in D(X)$, the poset of $\Mc$-subobjects of $\varphi$ in $\Synt_\cSFour(D)$ is isomorphic to ${\downarrow}\varphi \subseteq D(X)$.
\end{lemma}

\begin{proof}
	First of all, every $\psi \leq \varphi$ is an $\Mc$-subobject of $\varphi \in \Synt_{\cSFour}(D)$, the map $\psi\to\varphi$ being the restriction of the identity. If $\psi, \psi' \leq \varphi$ and if $\psi\to\varphi$ factors through $\psi'\to\varphi$ in $\Synt_{\cSFour}(D)$, then the same holds in $\Synt_{\cBA}(D)$ and thus $\psi\leq\psi'$. Consequently, ${\downarrow}\varphi$ embeds in the poset of $\Mc$-subobjects of $\varphi$. It remains to show that any $f : \psi\to\varphi$ in $\Mc$ is isomorphic to the restriction of an identity. If $f$ is in $\Mc$, it means that there are isomorphisms $u, v$ (continuous maps whose inverses are also continuous) making the following diagram commute, where $\varphi' \hookrightarrow \psi'$ is the restriction of an identity.
	\[\begin{tikzcd}
		\varphi \ar[r,"f",hook] \ar[d,"u"'] & \psi \ar[d,"v"]\\[-0.6em]
		\varphi' \ar[r,hook] & \psi'
	\end{tikzcd}\]
	We show that $f : \varphi\to\psi$ is isomorphic to the inclusion $v^{-1}(\varphi') \subseteq \psi$. Let $f' : \varphi \to v^{-1}(\varphi')$ be the restriction of $f$. It is continuous since $f$ is continuous. Let $v' : v^{-1}(\varphi') \to \varphi'$ be the restriction of $v$. It is continuous and has a continuous inverse. Then $v' \circ f' = u$ since $v\circ f = u$, and as a consequence $f'$ is the inverse of $u^{-1} \circ v'$.
\end{proof}

For any morphism $f : \varphi\to\psi$ in $\Synt_\cSFour(D)$, the map $f^{-1} : {\downarrow}\psi \to {\downarrow}\varphi$ is a morphism of Boolean algebras, by transferring that fact from $\Synt_\cBA(D)$. Thus we have shown:

\begin{proposition}
    $\Synt_\cSFour(D)$ is an f-Boolean category.
\end{proposition}

We will now equip $\Synt_\cSFour(D)$ with a structure of topological modal category. To avoid confusion, we write $\Diamond_X$ for the modality of $D(X)$. Given $\varphi \in D(X)$, we define a structure of S4 modal algebra on ${\downarrow}\varphi$ by $\Diamond_\varphi \psi = \varphi \land \Diamond_X \psi$. It is indeed an S4 modality:
\begin{itemize}
	\item $\Diamond_\varphi(\psi\lor\psi') = \varphi \land \Diamond_X (\psi \lor \psi') = \varphi \land [(\Diamond_X \psi) \lor (\Diamond_X \psi')] = (\Diamond_\varphi \psi) \lor (\Diamond_\varphi \psi')$
	\item If $\psi \leq \varphi$ then $\psi \leq \Diamond_X \psi$ and hence $\psi \leq \varphi \land \Diamond_X \psi = \Diamond_\varphi \psi$.
	\item $\varphi \land \Diamond_X (\varphi \land \Diamond_X \psi) \leq \varphi \land \Diamond_X \Diamond_X \psi = \varphi \land \Diamond_X \psi$
\end{itemize}

We conclude with the main result of this section.

\begin{proposition}
    $\Synt_\cSFour(D)$ is a topological modal category.
\end{proposition}

\begin{proof}
By the discussion above, $\Synt_\cSFour(D)$ is an f-Boolean category and each lattice of subobjects has a structure of S4 modal algebra.

We check that for any continuous map $f : \varphi\to\psi$, the map $f^{-1} : {\downarrow}\psi \to {\downarrow}\varphi$ is a lax morphism with respect to the modalities defined. We know that $f$ is the restriction of some $\tilde{f} : X\relto Y$ in $\Cc(D)$. By \eqref{ax:LC}, we have $\tilde{f} \Diamond_X \tilde{f}^{-1} \theta \leq \Diamond_Y \theta$ for all $\theta \subseteq Y$. If $\theta \subseteq \psi$ then $f \Diamond_\varphi f^{-1} \theta \leq \tilde{f} \Diamond_X \tilde{f}^{-1} \theta \leq \Diamond_Y \theta$ and moreover $f \Diamond_\varphi f^{-1} \theta \leq \psi$. This shows that $f \Diamond_\varphi f^{-1} \theta \leq \Diamond_\psi \theta$. By using the adjunction between $f[{-}]$ and $f^{-1}({-})$, we obtain that $f^{-1}$ is a lax morphism.

\paragraph{The axiom \eqref{ax:PI}} Let $\varphi \subseteq X$ and $\psi \subseteq Y$. Let $p : X\times Y\to X$ and $q : X\times Y\to Y$ be the two projections. Let $\tilde{p} : \varphi\times\psi \to \varphi$ and $\tilde{q} : \varphi\times\psi \to \psi$ be the restrictions of $p$ and $q$. Let $\theta \leq \varphi$ and $\xi \leq \psi$. Then
\begin{align*}
\tilde{p}^{-1}(\Diamond_\varphi \theta) \land \tilde{q}^{-1}(\Diamond_\psi \xi) &\leq p^{-1}(\Diamond_X \theta) \land q^{-1}(\Diamond_Y \xi)\\
&\leq \Diamond_{X\times Y} (p^{-1} \theta \land q^{-1} \xi)\text{.}
\end{align*}
Note that $p^{-1} \theta \leq \varphi\times Y$ and $q^{-1} \xi \leq X\times\psi$, so that $p^{-1} \theta \land q^{-1} \xi = \tilde{p}^{-1} \theta \land \tilde{q}^{-1} \xi$. Since also $\tilde{p}^{-1}(\Diamond_\varphi \theta) \land \tilde{q}^{-1}(\Diamond_\psi \xi) \leq \varphi\times\psi$, we get
\[ \tilde{p}^{-1}(\Diamond_\varphi \theta) \land \tilde{q}^{-1}(\Diamond_\psi \xi) \leq \Diamond_{\varphi\times\psi} (\tilde{p}^{-1} \theta \land \tilde{q}^{-1} \xi) \text{.} \]

\paragraph{The axiom \eqref{ax:LC}} We start by proving that \eqref{ax:LC} is satisfied for the acceptable loops whose construction involves only objects in $\cC$. Let $R_1,{\ldots},R_n$ be such a loop. Each application of \ref{acc:4} in the construction of $R_1,{\ldots},R_n$ uses a continuous partial map $f$, which is thus the restriction of some partial map $f'$ in $\Cc(D)$. If at each application of \ref{acc:4}, we replace $f$ by $f'$, then we obtain another acceptable loop $R'_1,{\ldots},R'_n$ such that $R_i \leq R'_i$ for each $i$. Since $R'_1,{\ldots},R'_n$ involves only partial maps from $\Cc(D)$, it satisfies \eqref{ax:LC} and we get
\[ \Diamond R_1 \cdots \Diamond R_n A \leq \Diamond R'_1 \cdots \Diamond R'_n A \leq \Diamond A \text{.} \]

Before showing that \eqref{ax:LC} holds for all the acceptable loops, we note that for every $\varphi \subseteq X$ and every $\psi \subseteq Y$, every partial map $f : \varphi\relto\psi$ in $\Synt_\cSFour(D)$ is the restriction of a partial map $X\relto Y$ in $\Cc(D)$. Indeed $f$ is a continuous map $\varphi' \to \psi$ with $\varphi' \leq \varphi$, and $\varphi'\to\psi$ is the restriction of a partial map in $\Cc(D)$. Let $R_1,{\ldots},R_n$ be an arbitrary acceptable loop in $\Synt_\cSFour(D)$. Let $\varphi(0), \varphi(1), {\ldots}, \varphi(n)$ be the associated sequence of domains and codomains, with $\varphi(0) = \varphi(n)$. For each $i$, we have some $X(i) \in \cC$ with $\varphi(i) \subseteq X(i)$. We can treat instead each $R_i$ as a relation $X(i) \relto X(i-1)$, and this produces another acceptable loop which satisfies \eqref{ax:LC}, as shown above. We thus have:
\[ \Diamond_{\varphi(0)} R_1 \cdots \Diamond_{\varphi(n-1)} R_n A \leq \Diamond_{X(0)} R_1 \cdots \Diamond_{X(n-1)} R_n A \leq \Diamond_{X(0)} A \text{.} \]
Since moreover $\Diamond_{\varphi(0)} R_1 \cdots \Diamond_{\varphi(n-1)} R_n A \leq \varphi(0)$, we obtain the desired inequality \eqref{ax:LC}.

We have checked that $\Synt_\cSFour(D)$ satisfies all the axioms of topological modal categories.
\end{proof}

\section{Conclusion}

In this paper we found necessary and sufficient conditions for a logical category endowed with interior and closure operators to be embeddable into a power of $\cTop$. The conditions we found (product independence and loop contraction principles)  derive from the definition of the product topology and its interaction with composition of continuous partial functions. The completeness proof goes through a construction of topological spaces obtained from lax successions of parallel pairs formed by relations and functions. Our results are meant to extend to first order logic classical results 
\cite{MckinseyTarski} 
connecting modal logic and topology.

Future work could be devoted to a better inspection of the definability power of the above formalism. Whereas in the propositional case modal logic cannot distinguish relevant classes of topological spaces from generic topological spaces, here the situation looks very different: for instance the $T_2$ separation axiom is expressible (just say that the diagonal is closed), Alexandroff spaces can be distinguished from arbitrary topological spaces~\cite{GhiModalitaCategorie1990}, compact Hausdorff spaces also exhibit interesting logical behaviors (e.g., continuous maps between them are closed), etc. The impression is that topological categories are a very rich framework to explore.

From another point of view, since our completeness theorem involves arbitrary theories, it makes sense to investigate the model theory of structures like for instance topological groups, rings, etc.: this might pave the way to  new research opportunities. 

\AtNextBibliography{\small}
{\printbibliography[
heading=bibintoc,
title={References}
]}

@book{JohSketchesElephantTopos2002,
  title = {Sketches of an Elephant: A Topos Theory Compendium},
  shorttitle = {Sketches of an Elephant},
  author = {Johnstone, Peter T.},
  year = {2002},
  month = sep,
  series = {Oxford Logic Guides},
  publisher = {Oxford University Press},
  address = {Oxford, New York}
}

@incollection{PitCategoricalLogic2001,
  title = {Categorical logic},
  booktitle = {Handbook of Logic in Computer Science. Volume 5: Algebraic and Logical Structures},
  author = {Pitts, Andrew},
  editor = {Abramsky, S. and Gabbay, Dov M. and Maibaum, T. S. E.},
  year = {2001},
  month = jan,
  publisher = {OUP Oxford}
}

@misc{GM2025,
	title={First-Order Modal Logic via Logical Categories}, 
	author={Silvio Ghilardi and Jérémie Marquès},
	year={2025},
	eprint={2504.02985},
	archivePrefix={arXiv}
}

@book{GhilardiZawadowski2011,
	author = {Ghilardi, S. and Zawadowski, M.},
	title = {Sheaves, Games, and Model Completions: A Categorical Approach to Nonclassical Propositional Logics},
	year = {2011},
	isbn = {9048160367, 9789048160365},
	edition = {1st},
	publisher = {Springer Publishing Company, Incorporated},
}

@phdthesis{GhiModalitaCategorie1990,
	title = {Modalit{\`a} e categorie},
	author = {Ghilardi, Silvio},
	year = {1990},
	school = {Universit{\`a} degli Studi di Milano}
}

@article{SgrInteriorOperatorLogic1980,
	title = {The Interior Operator Logic and Product Topologies},
	author = {Sgro, Joseph},
	year = {1980},
	journal = {Transactions of the American Mathematical Society},
	volume = {258},
	number = {1},
	eprint = {1998283},
	pages = {99--112},
	publisher = {American Mathematical Society}
}

@inproceedings{GhiMelModalTensePredicate1988,
	title = {Modal and tense predicate logic: Models in presheaves and categorical conceptualization},
	shorttitle = {Modal and tense predicate logic},
	booktitle = {Categorical Algebra and its Applications},
	author = {Ghilardi, S. and Meloni, G. C.},
	editor = {Borceux, Francis},
	year = {1988},
	pages = {130--142},
	publisher = {Springer},
	address = {Berlin, Heidelberg}
}

@article{ReyToposTheoreticApproachReference1991,
	title = {A Topos-Theoretic Approach to Reference and Modality},
	author = {Reyes, Gonzalo E.},
	year = {1991},
	journal = {Notre Dame Journal of Formal Logic},
	volume = {32},
	number = {3},
	pages = {359--391},
	publisher = {Duke University Press}
}

@article{MakReyCompletenessResultsIntuitionistic1995,
	title = {Completeness results for intuitionistic and modal logic in a categorical setting},
	author = {Makkai, M and Reyes, G. E},
	year = {1995},
	month = mar,
	journal = {Annals of Pure and Applied Logic},
	volume = {72},
	number = {1},
	pages = {25--101}
}

@article{AwoKisTopologyModalityTopological2008,
	title = {Topology and Modality: The Topological Interpretation of First-Order Modal Logic},
	shorttitle = {TOPOLOGY AND MODALITY},
	author = {Awodey, Steve and Kishida, Kohei},
	year = {2008},
	month = aug,
	journal = {The Review of Symbolic Logic},
	volume = {1},
	number = {2},
	pages = {146--166}
}

@inproceedings {viareggio,
	AUTHOR = {S. Ghilardi and G. Meloni},
	TITLE = {Relational and topological semantics for temporal and modal predicative logic}, 
	BOOKTITLE = {Nuovi problemi della logica e della scienza},
	VOLUME = {II},
	PUBLISHER = {CLUEB Bologna},
	YEAR = {1991},
	PAGES = {59–77},
}

@article{MaiRosElementaryQuotientCompletion2013,
	title = {Elementary Quotient Completion},
	author = {Maietti, Maria Emilia and Rosolini, Giuseppe},
	year = {2013},
	journal = {Theory and Applications of Categories},
	volume = {27},
	number = {17},
	pages = {445--463}
}

@article {Pasquali,
	AUTHOR = {Pasquali, F.},
	TITLE = {Remarks on the tripos to topos construction: comprehension,
	extensionality, quotients and functional-completeness},
	JOURNAL = {Appl. Categ. Structures},
	FJOURNAL = {Applied Categorical Structures. A Journal Devoted to
	Applications of Categorical Methods in Algebra, Analysis,
	Computer Science, Logic, Order and Topology},
	VOLUME = {24},
	YEAR = {2016},
	NUMBER = {2},
	PAGES = {105--119},
	ISSN = {0927-2852,1572-9095},
	MRCLASS = {18B25 (03G30 18F20)},
	MRNUMBER = {3474891},
	MRREVIEWER = {Ali\ Madanshekaf},
	DOI = {10.1007/s10485-014-9388-1},
	URL = {https://doi.org/10.1007/s10485-014-9388-1},
}

@article{McKinseyTarski,
         AUTHOR = {McKinsey, J.C.C and Tarski, A.},
         TITLE = {The Algebra of Topology},
         JOURNAL = {Ann. Math.},
         YEAR = {1944},
         VOLUME = {45},
	   NUMBER = {1},
         PAGES = {141--191},
}

@incollection {handbook,
	AUTHOR = {Bra\"uner, T. and Ghilardi, S.},
	TITLE = {First-order modal logic},
	BOOKTITLE = {Handbook of modal logic},
	SERIES = {Stud. Log. Pract. Reason.},
	VOLUME = {3},
	PAGES = {549--620},
	PUBLISHER = {Elsevier B. V., Amsterdam},
	YEAR = {2007},
	ISBN = {978-0-444-51690-9; 0-444-51690-5},
	MRCLASS = {03B45},
	MRNUMBER = {3618511},
}

\end{document}